%% file: main.tex
\documentclass{article}

\usepackage{mystyle} 
\makeatletter
\fancypagestyle{noHeading}{
        \fancyhead{}
        
}
\usepackage{fullpage}

\title{Higher Cheeger ratios of features in Laplace-Beltrami eigenfunctions}
\author{Gary Froyland and Christopher P.\ Rock}
\date{9 August 2023}

\hypersetup{pdfauthor={Gary Froyland and Christopher P. Rock},pdftitle={Extracting features from eigenfunctions via higher Cheeger constants}}

\graphicspath{{./Figures/}}

\newtheorem{theorem}{Theorem}[section]

\newtheorem{corollary}[theorem]{Corollary}
\newtheorem{proposition}[theorem]{Proposition}

\theoremstyle{definition}

\newtheorem{definition}[theorem]{Definition}
\newtheorem{example}{Example}
\newtheorem*{convexity-curvature}{Convexity-curvature condition}

\DeclareMathOperator{\Div}{div}

\DeclareMathOperator{\Leb}{Leb}
\DeclareMathOperator{\Span}{span}

\DeclareMathOperator*{\argmin}{arg\,min}

\DeclareMathOperator{\range}{range}
\DeclareMathOperator{\supp}{supp}
\DeclareMathOperator{\Int}{int}

\DeclareMathOperator{\sign}{sign}

\newcommand{\dd}{\,\mathrm{d}}
\newcommand{\ddd}{\mathrm{d}}

\newcommand{\e}{e}

\makeatletter
\def\imod#1{\allowbreak\mkern10mu({\operator@font mod}\,\,#1)}
\makeatother

\makeatletter
\newcommand{\ssdiff}{\mathbin{\vphantom{\triangle}\mathpalette\d@tD@lta\relax}}
\newcommand{\d@tD@lta}[2]{%
  \ooalign{\hidewidth$\m@th#1\cdot$\hidewidth\cr$\m@th#1\triangle$\cr}%
}
\makeatother

\newcommand{\pagelabel}[1]{\phantomsection\label{#1}}

\begin{document}
    \typeout{Started document}
    \maketitle
    \typeout{Finished title page}

\input{01_abstract.tex}

    \typeout{#############################^^J#############################^^J#############################^^J#############################^^J}
    \typeout{^^JFinished preliminary forms^^J^^J}
    \typeout{#############################^^J#############################^^J#############################^^J#############################^^J}
    
        
        \typeout{Modified page settings}

\input{02_introduction.tex}

\input{03_preliminaries.tex}

\input{04_cheeger.tex}

\input{05_examples.tex}

    \input{06_summary.tex}

       \section{Acknowledgements}
       The authors thank the anonymous referees for their constructive suggestions. The research of GF was partially supported by Australian Research Council Discovery Projects DP180101223 and DP210100357.  CPR's research was supported an Australia Research Training Program scholarship.

    \pagestyle{noHeading}
    \bibliographystyle{abbrv}
    \bibliography{master}

\end{document}

%% file: 01_abstract.tex
\begin{abstract}
    
This paper investigates links between the eigenvalues and eigenfunctions of the Laplace-Beltrami operator, and the higher Cheeger constants of smooth Riemannian manifolds, possibly weighted and/or with boundary. The higher Cheeger constants give a loose description of the major geometric features of a manifold. We give a constructive upper bound on the higher Cheeger constants, in terms of the eigenvalue of any eigenfunction with the corresponding number of nodal domains. Specifically, we show that for each such eigenfunction, a positive-measure collection of its superlevel sets have their Cheeger ratios bounded above in terms of the corresponding eigenvalue. 

Some manifolds have their major features entwined across several eigenfunctions, and no single eigenfunction contains all the major features. In this case, there may exist carefully chosen linear combinations of the eigenfunctions, each with large values on a single feature, and small values elsewhere. We can then apply a soft-thresholding operator to these linear combinations to obtain new functions, each supported on a single feature. We show that the Cheeger ratios of the level sets of these functions also give an upper bound on the Laplace-Beltrami eigenvalues. We extend these level set results to nonautonomous dynamical systems, and show that the dynamic Laplacian eigenfunctions reveal sets with small dynamic Cheeger ratios.

\end{abstract}

%% file: 02_introduction.tex

\section{Introduction}
The classical static \emph{Cheeger problem} is an optimisation problem in Riemannian geometry, which has been studied extensively in relation to the eigenvalues of the Laplace-Beltrami operator \cite{cheeger,buser,ledoux,milman}. 
Given an $n$-dimensional Riemannian manifold $(M,g)$ with volume measure $V$ and induced $n-1$-dimensional Hausdorff measure $V_{n-1}$, the \emph{Neumann Cheeger ratio} of a set $A \subset M$ with suitably smooth boundary is the ratio $\mathcal{J}_N(A):=\frac{V_{n-1}(\partial A \cap \Int M)}{V(A)}$. 
The Neumann Cheeger problem consists of finding a set that minimises $\mathcal{J}_N(A)$ over sets $A \subset M$ satisfying $V(A) \le \frac{V(M)}{2}$. The resulting minimal ratio is known as the \emph{Neumann Cheeger constant for $M$}.
For compact $n$-dimensional submanifolds $M \subset \mathbb{R}^n$, a Neumann Cheeger ratio minimiser is a set $A \subset M$ which is separated from $M\backslash \overline{A}$ by an optimal `bottleneck'. We give an example in Figure \ref{fig:bowtie-neumann}. 
The \emph{Dirichlet Cheeger ratio} of a set $A \subset M$ with suitably smooth boundary is the ratio $\mathcal{J}_D(A):=\frac{V_{n-1}(\partial A)}{V(A)}$, and the Dirichlet Cheeger problem consists of finding a set that minimises $\mathcal{J}_D(A)$ over subsets $A \subset M$. The resulting minimal ratio is known as the \emph{Dirichlet Cheeger constant for $M$}.
A Dirichlet Cheeger ratio minimiser is a region with an optimal balance between large volume and small boundary.  
For $n$-dimensional $M \subset \mathbb{R}^n$ endowed with the Euclidean metric and $A \subset M$, $\mathcal{J}_D(A)$ decreases by a factor of $s$ when we dilate $A$ by a factor of $s$ in each dimension, so minimisers for $\mathcal{J}_D(A)$ always contact $\partial M$ ({\cite[Theorem 3.5]{pariniintro}}). We give an example in Figure \ref{fig:bowtie-dirichlet}. 

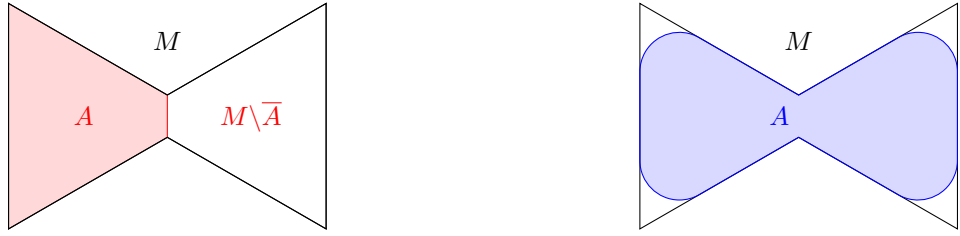
\begin{figure}
\centering
    \begin{subfigure}[t]{0.48\linewidth}
    \centering
    \begin{tikzpicture}
    \pgfmathsetmacro\sidelen{3-3*0.281284433486614*2/3}
    \pgfmathsetmacro\centerlen{\sidelen*sqrt(3)/2}
    \pgfmathsetmacro\complemlen{2*\centerlen-1}
    \draw (0,0) -- ++ (30:\sidelen) coordinate (gapi) -- ++ (330:\sidelen) coordinate (endi) -- ++ (90:3) coordinate (endii)
    -- ++ (210:\sidelen) coordinate (gapii) -- ++ (150:\sidelen) -- cycle (\centerlen,2.5) node {$M$};
    \fill[red,fill opacity=0.15] (0,0) -- (gapi) -- (gapii) -- (0,3) -- cycle;
    \draw[red] (gapi) -- (gapii) (1,1.5) node {$A$} (\complemlen,1.5) node {$M \backslash \overline{A}$};
    \draw (0,0) -- (gapi) -- (endi) -- (endii) -- (gapii) -- (0,3) -- cycle;
    \end{tikzpicture}
    \caption{The Neumann Cheeger ratio for $A \subset M$ is minimised when $A$ and $M \backslash \overline{A}$ are separated by the narrowest part of the bow-tie figure. }
    \label{fig:bowtie-neumann}
    \end{subfigure}
    \nolinebreak\quad
    \begin{subfigure}[t]{0.48\textwidth}
    \centering
    \begin{tikzpicture}
    \pgfmathsetmacro\sidelen{3-3*0.281284433486614*2/3}
    \pgfmathsetmacro\cornerlen{3*0.3034045210084015}
    \pgfmathsetmacro\arclen{\cornerlen/sqrt(3)}
    \pgfmathsetmacro\shortsidelen{\sidelen-\cornerlen}
    \pgfmathsetmacro\shortendlen{3-2*\cornerlen}
    \pgfmathsetmacro\centerlen{\sidelen*sqrt(3)/2}
    \draw (0,0) -- ++ (30:\sidelen) coordinate (endi) -- ++ (330:\sidelen) -- ++ (90:3) 
    -- ++ (210:\sidelen) coordinate (endii) -- ++ (150:\sidelen) -- cycle (\centerlen,2.5) node {$M$};
    \filldraw[blue,fill opacity=0.15] (0,\cornerlen) arc (180:300:\arclen) -- ++ (30:\shortsidelen) -- ++ (330:\shortsidelen) arc (240:360:\arclen) -- ++ (0,\shortendlen) 
    arc (0:120:\arclen) -- ++ (210:\shortsidelen) -- ++ (150:\shortsidelen) arc (60:180:\arclen) -- cycle;
    \node[left,blue] at (\centerlen,1.5) {$A$};
   \end{tikzpicture}
   \caption{The Dirichlet Cheeger ratio for $A \subset M$ is minimised when $A$ contains most of the open space of $M$, but only extends partway into the corners of $M$ \cite[Example 4.2]{leonardi-pratelli}. }
   \label{fig:bowtie-dirichlet}
    \end{subfigure}

\caption{Neumann and Dirichlet Cheeger minimisers for $M \subset \mathbb{R}^2$ equipped with the Euclidean metric.  
}
\label{fig:bowtie}
\end{figure}

The Cheeger problem can be extended to seek \emph{collections} of subsets, each of which have small Cheeger ratios. Given a collection of $k$ disjoint sets $A_1,\ldots,A_k \subset M$, the Neumann and Dirichlet Cheeger ratios of $\{A_1,\ldots,A_k\}$ are given by $\mathcal{J}_N(\{A_1,\ldots,A_k\}):=\max_{1 \le i \le k} \mathcal{J}_N(A_i)$ and $\mathcal{J}_D(\{A_1,\ldots,A_k\}):=\max_{1 \le i \le k} \mathcal{J}_D(A_i)$, respectively, i.e.\ the Cheeger ratio of a collection of disjoint subsets of $M$ is the maximum Cheeger ratio among the subsets. 
For each $k \ge 1$, the \emph{$k$th Neumann} or \emph{Dirichlet Cheeger problem} consists of finding a collection of $k$ disjoint sets $\{A_1,\ldots,A_k\}$ which minimises $\mathcal{J}_N(\{A_1,\ldots,A_k\})$ or $\mathcal{J}_D(\{A_1,\ldots,A_k\})$. 
The first Dirichlet Cheeger problem is exactly the classical Dirichlet Cheeger problem, while the second Neumann Cheeger problem corresponds to the classical Neumann Cheeger problem. 
The $k$th Cheeger problems for larger $k$ are called the \emph{higher Cheeger problems}, and the infima are called the \emph{higher Cheeger constants}.

Exact minimisers for the Cheeger problem have only been computed for a few sets or classes of sets (see e.g.\ \cite{leonardi-neumayer-saracco,bobkov-parini-21,leonardi15,caroccia-ciani}). In particular, \cite[Theorem 1.4]{leonardi-neumayer-saracco} obtains an expression for the Cheeger-minimising set of any subset of $\mathbb{R}^2$ without a `neck'. 
We are instead interested in using the Cheeger problem to identify necks, and the approach of \cite{leonardi-neumayer-saracco} does not extend to sets with necks (see e.g.\ \cite[Figs 1--2]{leonardi-neumayer-saracco}). 
There are some algorithms for solving Cheeger problems numerically (see e.g.\ \cite{carlier-comte-peyre,caselles-facciolo-meinhardt,buttazzo-carlier-comte,caboussat-glowinski-pons,lachand-robert-oudet}), but these algorithms apply only to the classical Cheeger problems, not the versions with $k \ge 2$ (in the Dirichlet case) or $k \ge 3$ (in the Neumann case). 
These algorithms have not been studied on Riemannian manifolds other than full-dimensional subsets of $\mathbb{R}^n$. 
Understanding the connectivity of more general Riemannian manifolds is important in settings such as manifold learning (e.g.\ \cite{coifman-lafon,izenman}), where one studies the geometry of a low-dimensional submanifold embedded in some high-dimensional Euclidean space. 
The second Dirichlet Cheeger problem is studied in \cite{bobkov-parini}, where the authors solve this problem for one specific subset of $\mathbb{R}^2$ (an annulus). 

Approximate minima and minimisers for the higher Cheeger problem, and upper bounds on the higher Cheeger constants, can be found using the eigenfunctions and eigenvalues of the (possibly \emph{weighted}) \emph{Laplace-Beltrami operator}. Miclo \cite{miclo15} and others have given upper bounds on the $k$th Cheeger constant on boundaryless manifolds, up to a non-explicit factor depending cubically on $k$. Miclo improves this dependence on $k$ to sub-logarithmic, by using (for example) the $2k$th eigenvalue to bound the $k$th Cheeger constant. We prove an alternative upper bound on the $k$th Cheeger constant (Theorem \ref{thm:cheeger}), extending a result from the graph setting \cite[Theorem 5]{daneshgar-hajiabolhassan-javadi}, in terms of the eigenvalue of any eigenfunction with $k$ or more nodal domains, up to a small constant factor independent of $k$. Thus, we can obtain a much tighter upper bound on the $k$th Cheeger constant whenever the appropriate eigenfunction has sufficiently many nodal domains. Our bound also applies to manifolds with nonempty boundary, under Neumann or Dirichlet boundary conditions. Moreover, our bound is constructive - we show that any (possibly weighted) Laplace-Beltrami eigenfunction has superlevel sets within each nodal domain whose Cheeger ratios are also bounded above. 
A similar approach is used in the graph setting in e.g.\ \cite[sec 1.1]{kllot}, to obtain a 2-partition of a graph with a low conductance from the first nontrivial graph Laplacian eigenvalue. Our approach is primarily useful in situations where Laplacian eigenfunctions on a manifold are calculated or approximated explicitly. 

An important question in the study of nonautonomous dynamical systems is how to divide the phase space into regions which interact minimally with each other. 
In purely deterministic dynamics, any two disjoint regions have no interaction with each other, so we instead consider regions whose boundaries remain small, relative to their size, as they evolve with the deterministic dynamics. The ratio of a region's time-averaged boundary size to its overall size is called its \emph{dynamic Cheeger ratio}. Sets with small dynamic Cheeger ratio are called \emph{coherent sets}, and the infimal dynamic Cheeger ratio is called the \emph{dynamic Cheeger constant} \cite{F15,FKw17}. We can obtain an upper bound on the dynamic Cheeger constants using the eigenvalues of an operator, which acts on the domain of the dynamical system, called the \emph{dynamic Laplacian}. 
We show that $k$ disjoint coherent sets with quality guarantees -- upper bounds on their dynamic Cheeger ratios -- can be obtained from any eigenfunction with $k$ nodal domains (Theorem \ref{thm:dcheeger}). 

The remainder of this article is structured as follows. In section \ref{sec:cheeger-preliminaries}, we provide some basic definitions and define the higher Cheeger constants. In subsections \ref{sec:cheegerintro}--\ref{sec:hcheeg}, we summarise prior upper bounds on the Cheeger constants in terms of Laplace-Beltrami eigenvalues. We also state our own constructive upper bounds, which depend on properties of the eigenfunctions (Theorem \ref{thm:cheeger} and Proposition \ref{thm:seba}). In subsection \ref{sec:dcheeg}, we generalise these results to the dynamic setting. Lastly, in section \ref{sec:cheeger-examples}, we give some examples comparing our bounds to bounds from the literature.

%% file: 03_preliminaries.tex

\section{Preliminaries} \label{sec:cheeger-preliminaries} 

\subsection{Higher Cheeger constants} \label{sec:cheegerdef}
Let $(M,g)$ be a smooth Riemannian manifold, possibly with nonempty boundary, i.e.\ a second-countable Hausdorff space where each point of $M$ has a neighbourhood diffeomorphic to a relatively open subset of $\{x \in \mathbb{R}^n:x_n\ge 0\}$. Except where otherwise noted, we assume all Riemannian manifolds are $n$-dimensional ($n \ge 2$), $C^\infty$, compact and connected, and have smooth boundary if they have a nonempty boundary. Let $V$ and $\ddd V$ denote the volume measure and volume form on $M$ induced by $g$. 
Let $(M,g,\mu)$ be a \emph{weighted manifold}, i.e.\ a Riemannian manifold $(M,g)$ equipped with a measure $\mu$ satisfying $\ddd\mu=\e^\phi \dd V$ for some $\phi \in C^\infty(M)$. 
Note that we can treat any Riemannian manifold as a weighted manifold by taking $\mu=V$, so all our results for weighted manifolds extend directly to unweighted manifolds (i.e.\ manifolds where $\phi=0$ everywhere). 
On each $n-1$-dimensional submanifold $\Sigma \subset M$, let $V_{n-1}$ and $\ddd V_{n-1}$ denote the $n-1$-dimensional Riemannian volume measure and volume form on $\Sigma$, and let $\mu_{n-1}$ be the measure satisfying $\ddd\mu_{n-1}:=\e^\phi \dd V_{n-1}$. 
 


For a set $A \subset M$, we let $\partial^M A$ denote the relative topological boundary of $A$ in $M$, i.e.\ the set of points $p \in M$ such that every neighbourhood of $p$ contains both points in $A$ and points in $M \backslash A$. 
For example, if $M:=\{(x,y) \in \mathbb{R}^2:x^2+y^2\le 1\}$ and $A:=\{(x,y) \in M : y>0\}$, then $\partial^M A$ consists of the interval $\{(x,0):-1 \le x\le 1\}$ but not the semicircle $\{(x,y) \in M:x^2+y^2=1\}$. 
We define the Neumann and Dirichlet Cheeger constants as follows.

\begin{definition} \label{def:packing}
Let $\mathscr{P}_N(M)$ denote the collection of nonempty, relatively open subsets $A \subset M$ such that $\partial^M A$ is a codimension-1, $C^\infty$ submanifold of $M$ with boundary $\partial (\partial^M A) = \partial^M A \cap \partial M$. Let $\mathscr{P}_D(M)$ denote the collection of nonempty, relatively open subsets $A \subset M$ such that  $\overline{A} \cap \partial M = \emptyset$, and $\partial A$ is a codimension-1, $C^\infty$ submanifold of $M$. Then for $k \ge 1$, a \emph{Neumann}, resp.\ \emph{Dirichlet} \emph{$k$-packing} is a set $\mathcal{A}_k := \{A_1,\ldots,A_k\}$ such that each $A_i \in \mathscr{P}_N(M)$, resp.\ $A_i \in \mathscr{P}_D(M)$, and the $A_i$ are pairwise disjoint. Let $\mathscr{P}_{k,N}(M)$, resp.\ $\mathscr{P}_{k,D}(M)$ denote the set of Neumann, resp.\ Dirichlet $k$-packings for $M$. 
\end{definition}

\begin{definition}[Higher Cheeger constants] \label{def:whcheeg} 
For $k \ge 1$, the \emph{Neumann Cheeger ratio} of a Neumann $k$-packing $\{A_1,\ldots,A_k\} \in \mathscr{P}_{k,N}(M)$ is 
\begin{align}
    \mathcal{J}_N(\{A_1,\ldots,A_k\})&:=\max_{1\le i \le k} \frac{\mu_{n-1}(\partial^M A_i)}{\mu(A_i)}. \label{eq:wcheegerratio} 
\end{align}
The \emph{Dirichlet Cheeger ratio} of a Dirichlet $k$-packing $\{A_1,\ldots,A_k\} \in \mathscr{P}_{k,D}(M)$ is 
\begin{align}
    \mathcal{J}_D(\{A_1,\ldots,A_k\})&:=\max_{1\le i \le k} \frac{\mu_{n-1}(\partial A_i)}{\mu(A_i)}. \label{eq:wcheegerratiod}
\end{align}
The \emph{$k$th Neumann} and \emph{Dirichlet Cheeger constants} of $M$ are 
\begin{align}
    h_{k,N}&:=\inf_{\{A_1,\ldots,A_k\} \in \mathscr{P}_{k,N}(M)} \mathcal{J}_N(\{A_1,\ldots,A_k\}) \label{eq:defcheegnk} \\
    h_{k,D}&:=\inf_{\{A_1,\ldots,A_k\} \in \mathscr{P}_{k,D}(M)} \mathcal{J}_D(\{A_1,\ldots,A_k\}). \label{eq:defcheegdk}
\end{align}
\pagelabel{page:wcheegb} 
\end{definition}

We will sometimes write $\mathcal{J}_N(A)$ and $\mathcal{J}_D(A)$ instead of $\mathcal{J}_N(\{A\})$ and $\mathcal{J}_D(\{A\})$ for convenience. By this definition, we always have $h_{1,N}=0$, aligning with our notation where $\lambda_{1,N}=0$. In the special case $\partial M = \emptyset$, we write $\mathcal{J}_\emptyset$ and $h_{k,\emptyset}$, respectively for $\mathcal{J}_N$ and $h_{k,N}$, respectively, and refer to $\mathcal{J}_\emptyset$ and $h_{k,\emptyset}$ as the \emph{boundaryless Cheeger ratio} and \emph{constant}. 

Our Dirichlet Cheeger constants generalises Cheeger's original constant for manifolds with boundary \cite{cheeger}, while our Neumann Cheeger constants generalise the  boundaryless Cheeger constant of \cite{cheeger}, the Neumann Cheeger constant of \cite{buser80a}, and the $k$th boundaryless Cheeger constants \cite{miclo15} for $k \ge 1$. Our $h_{k,\emptyset}$ is exactly that defined in \cite[p.326]{miclo15}: Miclo requires that the $A_i$ are connected and that each connected component of $M \backslash A_i$ contains some $A_j$ for $j \ne i$, but Miclo notes that this does not change the value of $h_{k,\emptyset}$. 

Cheeger \cite{cheeger} and Buser \cite{buser80a} (see also \cite[p.499]{yau}) consider $h_{k,\emptyset}$ and $h_{k,N}$ for $k=2$ only, and they require that $\{A_1,A_2\}$ are a 2-partition for $M$ (up to sets of measure zero) with $\partial^M A_1 = \partial^M A_2$, instead of allowing 2-packings of $M$. This does not affect the value of $h_{2,N}$. To see this, choose any $\{A_1,A_2\} \in \mathscr{P}_{2,N}(M)$ with $\mu_{n-1}(\partial^M A_1) \le \mu_{n-1}(\partial^M A_2)$, and define the 2-packing $\{\tilde{A}_1,\tilde{A}_2\}$ by $\tilde{A}_1:=\overline{A_1}\backslash \partial^M \overline{A_1}$, $\tilde{A}_2:=M \backslash \overline{A_1}$. Then $\partial^M \tilde{A}_1=\partial^M \tilde{A}_2$, and $\{\tilde{A}_1,\tilde{A}_2\}$ is a 2-partition for $M$. The fact $\partial^M \tilde{A}_1 \subset \partial^M A_1$ implies $\{\tilde{A}_1,\tilde{A}_2\}\in \mathscr{P}_{2,N}(M)$, and since $\mu_{n-1}(\partial^M \tilde{A}_2) \le \mu_{n-1}(\partial^M A_1) \le \mu_{n-1}(\partial^M A_2)$ and $\mu(\tilde{A}_2)\ge \mu(A_2)$, we have $\mathcal{J}_N(\{\tilde{A}_1,\tilde{A}_2\}) \le \mathcal{J}_N(\{A_1,A_2\})$. 

Our Cheeger constants are defined slightly differently from those in \cite{bobkov-parini,de-ponti-mondino}, who take the infimum over arbitrary packings of $M$ and use \emph{perimeter} instead of Hausdorff measure. Bobkov and Parini's Cheeger constant is equal to $h_{k,D}$ for unweighted full-dimensional submanifolds of $\mathbb{R}^n$ (\cite[Proposition 3.6]{bobkov-parini} and e.g.\ \cite[Proposition 3.62]{ambrosio-fusco-pallara}), while $h_{2,N}$ gives an upper bound on de Ponti and Mondino's Cheeger constant on unweighted Riemannian manifolds by \cite[Proposition 2.37]{volkmann} (\cite{volkmann} defines perimeter differently to \cite{de-ponti-mondino}, but they are equal on unweighted manifolds by e.g.\ \cite[remark on Definition 2.33 and Theorems 2.38--2.39]{volkmann}). Yau \cite[p.499]{yau} also defines a variant of $h_{2,N}$ which does not require each $\partial A_i$ to be smooth. 
The paper \cite{saracco-stefani}, published after the first version of this manuscript was submitted, extends much of the work in \cite{bobkov-parini} to minimisers of more general functionals than the maximal Cheeger ratio, and to the more general settings described in \cite{francheschi-pinamonti-saracco-stefani}.

\subsection{Eigenvalues of the weighted Laplace-Beltrami operator}

Let $W^{1,2}(M;\mu)$ denote the Sobolev space of $L^2(M;\mu)$ functions $f$ with $L^2(M;\mu)$-integrable weak derivatives $\nabla f$, and let $W^{1,2}_0(M;\mu)$ denote the completion in the Sobolev norm ${\|\cdot\|}_{W^{1,2}(M;\mu)}^2:=\|\cdot\|_{L^2(M;\mu)}^2+{\||\nabla \cdot|\|_{L^2(M;\mu)}^2}$ of the set of $C^\infty(M)$ functions with compact support in $\Int M$ (see e.g.\ \cite[pp.14-15]{chavel84}). 

For any $C^1$ vector field $X$ on $M$, let $\Div X$ denote the \emph{divergence} of $X$ with respect to $\ddd V$ (defined in e.g.\ \cite[p.96]{grigor'yan05} or \cite[Prop.\ III.7.1 and proof]{chavel06}). Writing the Radon-Nikodym derivative of $\mu$ as $\ddd\mu=\e^\phi \dd V$, let $\Div_\mu X$ denote the \emph{weighted divergence} $\Div_\mu X:=\e^{-\phi}\Div(\e^\phi X)$ (see e.g.\ \cite[p.96]{grigor'yan05}). 
Then the \emph{weighted Laplace-Beltrami operator} $\Delta_\mu$ is defined for $f \in C^2(M)$ by 
\begin{align}
    \Delta_\mu f:=\Div_\mu \nabla f = \e^{-\phi} \Div(\e^\phi \nabla f). \label{eq:deflaplacew}
\end{align}

We consider the \emph{Neumann} and \emph{Dirichlet eigenproblems} for $\Delta_\mu$. The \emph{Neumann eigenproblem} is as follows: find $u \in C^\infty(M)$ and $\lambda \in \mathbb{R}$, such that  
\begin{align}
    \Delta_\mu u = \lambda u, \label{eq:laplacian}
\end{align}
subject to the \emph{Neumann boundary condition} (if $\partial M \not= \emptyset$) 
\begin{align}
    \frac{\partial u}{\partial \mathbf{n}} = 0 \quad \text{on } \partial M, \label{eq:neumann}
\end{align}
where $\mathbf{n}$ denotes the outward unit normal to $\partial M$. 
Solutions $u$ and $\lambda$ are called \emph{eigenfunctions} and \emph{eigenvalues} of $\Delta_\mu$. There is an orthogonal Schauder basis for $L^2(M;\mu)$ consisting of eigenfunctions of \eqref{eq:laplacian} satisfying \eqref{eq:neumann} (see e.g.\ \cite[Theorem 4.3.1]{lablee} or \cite[ch.\ III, Theorem 18]{berard86}). The corresponding eigenvalues form a non-positive decreasing sequence accumulating only at $-\infty$ (see e.g.\ \cite[Theorem 4.3.1]{lablee} or \cite[Theorems 11.5.1-11.5.2]{jost}). We denote the eigenvalues as $0=\lambda_{1,N} > \lambda_{2,N} \ge \lambda_{3,N} \ge \ldots$, or as $0=\lambda_{1,\emptyset} >  \lambda_{2,\emptyset}\ge \lambda_{3,\emptyset} \ge \ldots$ in the special case $\partial M=\emptyset$. The eigenvalue ordering induces an ordering on the corresponding eigenfunctions, so we will occasionally write the basis of eigenfunctions as $u_1,u_2,\ldots$. 

The \emph{Dirichlet eigenproblem} consists of finding $u\in C^\infty(M)$ and $\lambda\in \mathbb{R}$ which solves \eqref{eq:laplacian}, subject to the \emph{Dirichlet boundary condition},
\begin{align}
    u = 0 \quad \text{on } \partial M. \label{eq:dirichlet}
\end{align}
We assume $\partial M \not= \emptyset$ when we consider Dirichlet boundary conditions. There is also an orthogonal Schauder basis for $L^2(M;\mu)$ of eigenfunctions of \eqref{eq:laplacian} satisfying \eqref{eq:dirichlet}. In this case, the eigenvalues form a strictly negative decreasing sequence accumulating only at $-\infty$, and we denote them $0 > \lambda_{1,D} > \lambda_{2,D} \ge \lambda_{3,D} \ge \ldots$. 

The eigenvalues of $\Delta_\mu$ have the following variational characterisation (the proof of \cite[p.16]{chavel84} extends directly to the weighted case). 
\begin{theorem} \label{thm:variational}
Let $(M,g,\mu)$ be a weighted manifold, and let $u_1,u_2,\ldots$ denote a complete orthogonal basis of Neumann (resp.\ Dirichlet) eigenfunctions of $\Delta_{{\mu}}$ corresponding to $\lambda_{1,N},\lambda_{2,N},\ldots$ (resp.\ $\lambda_{1,D},\lambda_{2,D},\ldots$). Then for each $k \ge 1$, we have 
\begin{align}
    \lambda_{k,N} = -\inf_{\substack{f \in W^{1,2}(M) \\ \int_M u_i f \dd\mu=0, \forall i \in \{1, \ldots, k-1\}}} \frac{\||\nabla f|\|_{L^2(M;\mu)}^2}{\|f\|_{L^2(M;\mu)}^2}, \label{thm:variationaln}
\end{align}
resp.\  
\begin{align} 
    \lambda_{k,D} = -\inf_{\substack{f \in W^{1,2}_0(M) \\ \int_M u_i f\dd\mu=0, \forall i \in \{1, \ldots, k-1\}}} \frac{\||\nabla f|\|_{L^2(M;\mu)}^2}{\|f\|_{L^2(M;\mu)}^2},  \label{thm:variationald}
\end{align}
with equality if and only if $f$ is a Neumann (resp.\ Dirichlet) eigenfunction of $\Delta_\mu$ with eigenvalue $\lambda_{k,N}$ (resp.\ $\lambda_{k,D}$). 
\end{theorem}

A \emph{nodal domain} of a function $f \in C^0(M)$ is a maximal connected component of $M$ where $f$ is positive or negative. The number of nodal domains in the $k$th eigenfunction of $\Delta_\mu$ under Dirichlet or Neumann boundary conditions is bounded above by $k$. Courant \cite[p.452]{courant} proves this bound assuming each nodal domain has piecewise smooth boundary. Chavel \cite[pp.19--23]{chavel84} gives a proof in the boundaryless and Dirichlet cases which avoids the piecewise smooth boundary requirement via an approximation argument. Using a more general version of Green's formula \cite[Proposition 5.8 and remark after Proposition 5.10]{hofmann-mitrea-taylor}, we prove Courant's nodal domain theorem in the Neumann case without the piecewise-smooth boundary assumption, since we could not readily find this in the literature. 

\begin{theorem}[Courant's nodal domain theorem] \label{thm:courant}
Let $(M,g,\mu)$ be a weighted manifold. Then the $k$th Neumann or Dirichlet eigenfunction $u_k$ of $\Delta_\mu$ on $M$ has at most $k$ nodal domains. 
\end{theorem}

\begin{proof}
We prove only the Neumann case; the proof in \cite[pp.19-23]{chavel84} for the Dirichlet case extends immediately to weighted manifolds. 
Let $G_1,\ldots,G_k,G_{k+1},\ldots$ denote the nodal domains of $u_k$. For each $j=1,\ldots,k$, define $\psi_j \in W^{1,2}(M;\mu)$ by 
\begin{align*}
    \psi_j:=\begin{cases} u_k|_{G_j}, &\text{on } G_j, \\ 0, &\text{elsewhere}. \end{cases}
\end{align*}
Using Chavel's approximation argument \cite[pp.21--22]{chavel84} and the version of Green's formula in \cite[Proposition 5.8 and remark after Proposition 5.10]{hofmann-mitrea-taylor}, as in \eqref{eq:nablau}--\eqref{eq:lambdauu} below, for each $j$ we have $\frac{\| |\nabla \psi_j| \|_{L^2(G_j;\mu)}^2}{\|\psi_j\|_{L^2(G_j;\mu)}^2}=-\lambda_{k,N}$. 
One can select constants $\alpha_1,\ldots,\alpha_k \in \mathbb{R}$, not all zero, such that 
\begin{align*}
    f := \sum_{j=1}^k \alpha_j \psi_j
\end{align*}
satisfies 
\begin{align*}
    \int_M u_i f \dd\mu = 0,
\end{align*}
for each $i=1,\ldots,k-1$ (see e.g.\ \cite[p.17]{chavel84}). Noting that the $\psi_j$ are disjointly supported, we have 
\begin{align*}
    \frac{\||\nabla f|\|_{L^2(M;\mu)}^2}{\|f\|_{L^2(M;\mu)}^2} = \frac{\sum_{j=1}^k \alpha_j^2 \||\nabla \psi_j|\|^2_{L^2(M;\mu)}}{\sum_{j=1}^k \alpha_j^2 \|\psi_j\|_{L^2(M;\mu)}^2} ={-} \frac{\lambda_{k,N} \sum_{j=1}^k \alpha_j^2 \|\psi_j\|_{L^2(M;\mu)}^2}{\sum_{j=1}^k \alpha_j^2 \|\psi_j\|_{L^2(M;\mu)}^2} ={-} \lambda_{k,N}. 
\end{align*}
Thus, Theorem \ref{thm:variational} implies $f$ is an eigenfunction of $\Delta_\mu$ with eigenvalue $\lambda_{k,N}$ vanishing identically on $G_{k+1}$. But then Aronszajn's unique continuation principle \cite{aronszajn} implies that $f$ vanishes identically on $M$, which is a contradiction. 
\end{proof}

%% file: 04_cheeger.tex
\section{Classical and higher Cheeger inequalities} \label{sec:cheeger-section}

\subsection{Cheeger inequalities for the first nonzero eigenvalue} \label{sec:cheegerintro}

The classical Cheeger inequalities provide an explicit bound away from 0 for $\lambda_{1,D}$ or $\lambda_{2,N}$, in terms of $h_{1,D}$ or $h_{2,N}$. Cheeger \cite{cheeger} proves the boundaryless and Dirichlet cases, while Maz'ya \cite{maz'ya} (summarised in English in e.g.\ \cite[Sec.\ 6]{grigor'yan99}) independently proves a slightly stronger result some years prior. Yau \cite[Sec.\ 5, Corollary 1]{yau}, and later Buser \cite[Theorem 1.6]{buser80a}, prove the Neumann case. The Cheeger inequality can also be extended to metric measure spaces (including weighted manifolds). 
De Ponti and Mondino \cite[Theorem 3.6]{de-ponti-mondino} and Funano \cite[Lemma 7.1]{funano} give variants of the Cheeger inequality for metric spaces (including weighted manifolds), with a Rayleigh quotient in place of an eigenvalue. 

 Several other upper bounds on eigenvalues of $\Delta$ exist, which do not depend on the Cheeger constant (see for example \cite{hassannezhad} and references therein). Fewer bounds exist on the Cheeger constant: Ledoux \cite[Theorem 5.3]{ledoux} and Milman \cite[Theorem 1.5]{milman} have obtained bounds on the Cheeger constant in terms of concentration inequalities, while Dai et al \cite[Theorem 1.4]{dai} have obtained an upper bound on the Cheeger constant on convex manifolds in terms of the manifold's dimension, Ricci curvature and diameter. 

\begin{theorem}[Cheeger's inequality] \label{thm:firstcheeger} \hfill 

\begin{itemize}
\item \cite{cheeger}: Let $(M,g)$ be an unweighted, boundaryless, compact smooth Riemannian manifold. Then
\begin{align}
    \lambda_{2,\emptyset} \le -\frac{1}{4} h_{2,\emptyset}^2. \label{eq:boundarylessstaticcheeger}
\end{align}

\item \cite{cheeger,maz'ya} (see also \cite[Sec.\ 6]{grigor'yan99}): Let $(M,g)$ be an unweighted, connected, compact smooth Riemannian manifold with nonempty, smooth boundary. Then
\begin{align}
    \lambda_{1,D} \le -\frac{1}{4} h_{1,D}^2. \label{eq:dirichletstaticcheeger}
\end{align}

\item \cite[Sec.\ 5, Corollary 1]{yau}, \cite[Theorem 1.6]{buser80a}: Let $(M,g)$ be an unweighted, compact smooth Riemannian manifold with nonempty, smooth boundary.  Then 
\begin{align}
    \lambda_{2,N} \le -\frac{1}{4} h_{2,N}^2. \label{eq:neumannstaticcheeger}
\end{align}
\end{itemize}
\end{theorem}
The proofs of equations \eqref{eq:boundarylessstaticcheeger}-\eqref{eq:neumannstaticcheeger} all proceed similarly, using the coarea formula to relate the geometry of the level sets of an eigenfunction $u$ to its Dirichlet energy $\||\nabla u|\|^2$. Similar techniques are used in e.g.\ \cite{talenti,grigor'yan-netrusov-yau,klartag}. These results extend directly to weighted manifolds, and even to more general metric measure spaces (see e.g.\ \cite[Theorem 3.6]{de-ponti-mondino}), as well as to eigenvalues of the $p$-Laplace operator (e.g.\ \cite{matei,lefton-wei}). {On boundaryless manifolds with \emph{Ricci curvature} bounded below by $-K\le 0$, the leading eigenvalue $\lambda_{2,\emptyset}$ and the Cheeger constant $h_{2,\emptyset}$ satisfy $\lambda_{2,\emptyset} \ge -2.2\sqrt{K}h_{2,\emptyset} - 4.4h_{2,\emptyset}^2)$ (see \cite{buser,ledoux94,ledoux,de-ponti-mondino}, stated here with the constants from \cite{de-ponti-mondino}). Thus, $\lambda_{2,\emptyset}$ and $h_{2,\emptyset}$ are comparable,  in the sense that if one approaches zero then the other must also approach zero \cite{buser80a}.}


We prove that some of the superlevel sets within any nodal domain of any eigenfunction of $\Delta$ have an upper bound on their Cheeger ratio, in terms of the corresponding eigenvalue (Theorem \ref{thm:levelset}). This yields a constructive version of Theorem \ref{thm:firstcheeger} (Corollary \ref{thm:firstcheegerw}), and also allows us to prove a constructive higher Cheeger inequality (Theorem \ref{thm:cheeger}). 

For any nodal domain $G$ of a function $f\in C^0(M)$, we let $\range(f^2|_G):=\{s^2 : s \in f(G)\}$, and for any $s \in \range(f^2|_G)$, we define the \emph{$s$-superlevel set} of $f^2$ on $G$ as 
\begin{align}
    G_s:=\{p \in G : f(p)^2>s\}. \label{eq:Gsdef}
\end{align}

\begin{theorem} \label{thm:levelset}
Let $(M,g,\mu)$ be an $n$-dimensional weighted manifold. Let $u$ be some nonconstant Neumann, resp.\ Dirichlet, eigenfunction of $\Delta_\mu$, with eigenvalue $\lambda$. Let $G \subset M$ be any nodal domain of $u$. Then 
the set 
\begin{align}
    S_G := \mleft\{s \in \range(u^2|_G) : G_s \in \mathscr{P}_N(M), \lambda \le -\frac{1}{4}\mathcal{J}_N(G_s)^2 \mright\}, \label{eq:Sdef}
\end{align}
resp.\ 
\begin{align}
    S_G :=  \mleft\{s \in \range(u^2|_G) : G_s \in \mathscr{P}_D(M), \lambda \le -\frac{1}{4}\mathcal{J}_D(G_s)^2 \mright\}, \label{eq:Sdefd}
\end{align}
has positive Lebesgue measure satisfying the lower bound {
\begin{align}
    \Leb(S_G) \ge \frac{\| \bar{h} - \mathrm{h} \|_{L^1(\range(u^2|_G);\mathbb{P})} \|u\|_{L^2(G;\mu)}^2}{2 (\bar{h} - \inf_{s \in \range(u^2|_G)} \mathrm{h}(s))\mu(G)}, \label{eq:Smeasure}
\end{align}
where $\mathrm{h}$, $\bar{h}$ and $\mathbb{P}$ are defined in the proof below.} 
\end{theorem}
 
\begin{proof}
We prove only the Neumann case; the Dirichlet case follows similarly. Firstly, we use the coarea formula to find an expression \eqref{eq:gradh} for the weighted average \eqref{eq:hdef} of $\mathcal{J}_N(G_s)$. Secondly, we use a Rayleigh quotient argument to bound $\lambda_{2,N}$ in terms of this weighted average (equation \eqref{eq:cheegerbarbound}). Lastly, we obtain our lower bound on the measure of $S_G$. 

The coarea formula (see e.g.\ \cite[13.4.2]{burago-zalgaller}) implies 
\begin{align}
    \int_G |\nabla (u^2)| \dd \mu = \int_{\range(u^2|_G)} \mu_{n-1}(\{p \in G : u^2(p)=s\}) \dd s. \label{eq:coarea}
\end{align}
It follows immediately from Sard's theorem (e.g.\ \cite[Theorem 6.10]{lee}), \cite[Theorem 6.2.8]{mukherjee} and the reasoning for \cite[Lemma 6.2.7]{mukherjee} that $G_s \in \mathscr{P}_N(M)$ and $\partial^M G_s=\{p \in G : u(p)^2=s\}$ for almost every $s \in \range(u^2|_G)$. 
For such $s$, we have $\mu_{n-1}(\{p \in G : u(p)^2=s\}) = \mu_{n-1}(\partial^M G_s)=\mathcal{J}_N(G_s) \mu(G_s)$, by the definition \eqref{eq:wcheegerratio}. Hence, we have 
\begin{align}
    \int_G |\nabla (u^2)| \dd \mu = \int_{\range(u^2|_G)} \mathcal{J}_N(G_s) \mu(G_s) \dd s. \label{eq:cheegerbarproof}
\end{align}
Define 
\begin{align}
    \bar{h}:=\frac{1}{\|u\|_{L^2(G;\mu)}^2} \int_{\range(u^2|_G)} \mathcal{J}_N(G_s) \mu(G_s) \dd s, \label{eq:hdef}
\end{align}
then $\bar{h}$ is the weighted average of $\mathcal{J}_N(G_s)$ over $\range(u^2|_G)$, according to the probability measure $\mathbb{P}$ on $\range(u^2|_G)$ given by $\mathbb{P}(L):=\int_L \frac{\mu(G_s)}{\|u\|_{L^2(G;\mu)}^2} \dd s$. Then \eqref{eq:cheegerbarproof} and \eqref{eq:hdef} yield 
\begin{align}
    \frac{\int_G |\nabla (u^2)| \dd \mu}{\|u\|_{L^2(G;\mu)}^2} =  \bar{h}. \label{eq:gradh}
\end{align}
Now, the Cauchy-Schwarz inequality implies 
\begin{align}
    2 \||\nabla u|\|_{L^2(G;\mu)} \|u\|_{L^2(G;\mu)} \ge 2\int_G u|\nabla u| \dd\mu = \int_G |\nabla(u^2)| \dd\mu. \label{eq:cauchyschwarz}
\end{align}
Using \eqref{eq:cauchyschwarz} and \eqref{eq:gradh}, we obtain 
\begin{align}
    \frac{\||\nabla u|\|_{L^2(G;\mu)}^2}{\|u\|_{L^2(G;\mu)}^2} \ge \frac{\mleft( \int_G |\nabla (u^2)| \dd \mu \mright)^2}{4\|u\|_{L^2(G;\mu)}^4} = \frac{1}{4} \bar{h}^2. \label{eq:rayleigh-h}
\end{align}

We can write $\||\nabla u|\|_{L^2(G;\mu)}^2$ as 
\begin{align}
    \||\nabla u|\|_{L^2(G;\mu)}^2=\int_G(\nabla u \cdot \nabla u) \dd\mu=\int_G \nabla u \cdot (\e^\phi \nabla u) \dd V. \label{eq:nablau}
\end{align}
Applying Green's formula (e.g.\ \cite[Proposition 5.8 and remark after Proposition 5.10]{hofmann-mitrea-taylor}) to $\e^\phi u \nabla u$ on $G$ via a short approximation argument%
\footnote{We apply Green's formula via an approximation argument, similarly to e.g.\ \cite[pp21--22]{chavel84}. We showed above that $G_s \in \mathscr{P}_N(M)$ for almost every $s \in \range(u^2|_G)$, but it does not follow that $G \in \mathscr{P}_N(M)$, or that $G$ has locally finite perimeter. Instead, choose some sequence $s_1,s_2,\ldots \in \range(u^2|_G)$ converging to 0, such that $G_{s_j} \in \mathscr{P}_N(M)$ for each $j$. Then taking $u_j:=u-s_j$ and applying Green's formula to $u_j \e^\phi \nabla u_j$ on $G_{s_j}$, and recalling \eqref{eq:deflaplacew}, yields $\int_{G_{s_j}} \nabla u_j \cdot (\e^\phi \nabla u_j) \dd V = -\int_{G_{s_j}} u_j \cdot \Delta_\mu u_j \dd \mu + \int_{(\partial^M G_{s_j}) \cup (\partial M \cap G_{s_j})} u_j \frac{\partial u_j}{\partial \mathbf{n}} \dd \mu_{n-1}$, where $\mathbf{n}$ is an outward unit normal to $\partial M$ or $\partial^M G_{s_j}$. But $u_j=0$ on $\partial^M G_{s_j}$ and $\frac{\partial u_j}{\partial\mathbf{n}}=0$ on $\partial M \cap G_{s_j}$, so the second integral disappears, and taking $j \to \infty$, we obtain \eqref{eq:nodalrayleigh}. 
}, %
recalling \eqref{eq:deflaplacew} and noting that $u=0$ on $\partial^M G$ and $\frac{\partial u}{\partial \mathbf{n}}=0$ on $\partial G \cap \partial M$ (where $\mathbf{n}$ denotes the outward normal of $M$), we obtain 
\begin{align}
    \int_G \nabla u \cdot (\e^\phi \nabla u) \dd V = -\int_G u \cdot \Delta_\mu u \dd\mu + 0. \label{eq:nodalrayleigh}
\end{align}
Since $u \cdot \Delta_\mu u=\lambda u^2$, we have 
\begin{align}
    -\int_G u \cdot \Delta_\mu u \dd\mu = -\lambda \|u\|_{L^2(G;\mu)}^2. \label{eq:lambdauu}
\end{align}
Hence \eqref{eq:nablau}--\eqref{eq:lambdauu} and \eqref{eq:rayleigh-h} imply 
\begin{align}
    \lambda = -\frac{\||\nabla u|\|_{L^2(G;\mu)}^2}{\|u\|_{L^2(G;\mu)}^2} \le -\frac{1}{4} \bar{h}^2. \label{eq:cheegerbarbound}
\end{align}
But $\bar{h}$ is a weighted average over $s \in \range(u^2|_G)$ of $\mathcal{J}_N(G_s)$, so the set $S_G' := \{s \in \range(u^2|_G) : \mathcal{J}_N(G_s) \le \bar{h} \}$ has positive measure. By \eqref{eq:cheegerbarbound} and the definition \eqref{eq:Sdef}, we have $S_G' \subseteq S_G$, so $S_G$ must also have positive measure. 

We can put a lower bound on the measure of $S_G$, as follows. Let $\mathrm{h}(s):=\mathcal{J}_N(G_s)$. Then we have
\begin{align*}
    \int_{S_G'} (\bar{h}-\mathrm{h}(s)) \frac{\mu(G_s)}{\|u\|_{L^2(G;\mu)}^2} \dd s = \int_{S_G'} (\bar{h}-\mathrm{h}(s)) \dd \mathbb{P}(s) = \frac{\| \bar{h} - \mathrm{h} \|_{L^1(\range(u^2|_G);\mathbb{P})}}{2}, 
\end{align*}
and 
\begin{align*}
    \int_{S_G'} (\bar{h}-\mathrm{h}(s)) \frac{\mu(G_s)}{\|u\|_{L^2(G;\mu)}^2} \dd s 
    &\le \int_{S_G'} \frac{\mu(G_s)}{\|u\|_{L^2(G;\mu)}^2} \dd s \, \mleft(\bar{h} - \inf_{s \in \range(u^2|_G)} \mathrm{h}(s)\mright) \\
    &\le \Leb(S_G') \frac{\mu(G)}{\|u\|_{L^2(G;\mu)}^2} \mleft(\bar{h} - \inf_{s \in \range(u^2|_G)} \mathrm{h}(s)\mright) \\
    &\le \Leb(S_G) \frac{\mu(G)}{\|u\|_{L^2(G;\mu)}^2} \mleft(\bar{h} - \inf_{s \in \range(u^2|_G)} \mathrm{h}(s)\mright),
\end{align*}
{which we can rearrange to yield \eqref{eq:Smeasure}.}
A similar result holds in the Dirichlet case, replacing $\mathcal{J}_N$ with $\mathcal{J}_D$ in the definition of $\bar{h},\mathrm{h},\mathbb{P}$, and noting that $\overline{G_s} \cap \partial M=0$ for all $s \ne 0$. 
\end{proof}

\begin{corollary} \label{thm:firstcheegerw}
Let $(M,g,\mu)$ be a weighted manifold. For each Neumann eigenfunction $u$ corresponding to $\lambda_{2,N}$, there is a nodal domain $G$ of $u$ such that the set $S_G$ defined in \eqref{eq:Sdef} has positive measure {and satisfies \eqref{eq:Smeasure}}, and for each $s \in S_G$, defining $G_s$ as in \eqref{eq:Gsdef}, the 2-packing $\{G_s,M \backslash \overline{G_s}\}$ satisfies 
\begin{align}
    \lambda_{2,N} \le -\frac{1}{4}\mathcal{J}_N(\{G_s,M \backslash \overline{G_s}\})^2. \label{eq:firstcheegerneumannls}
\end{align}
If $\partial M \not= \emptyset$, there is a unique Dirichlet eigenfunction $u$ corresponding to $\lambda_{1,D}$ (up to scaling), and this $u$ has only a single nodal domain $G = M \backslash \partial M$. The set $S_G$ defined in \eqref{eq:Sdefd} has positive measure, and for each $s \in S_G$, the set $G_s$ defined in \eqref{eq:Gsdef} satisfies 
\begin{align} 
    \lambda_{1,D} \le -\frac{1}{4} \mathcal{J}_D(G_s)^2. \label{eq:firstcheegerdirichletls} 
\end{align} 
\end{corollary} 

\begin{proof}
By Theorem \ref{thm:courant} and e.g.\ \cite[Propositions 4.5.8--4.5.9]{lablee}, the eigenfunction corresponding to $\lambda_{1,D}$ has one nodal domain $G=M \backslash \partial M$, while each eigenfunction corresponding to $\lambda_{2,N}$ has two nodal domains, and Theorem \ref{thm:levelset} immediately yields \eqref{eq:firstcheegerdirichletls}. 
In the Neumann case, let $G$ denote whichever nodal domain of $u$ satisfies $\mu(G)\le \mu(M \backslash \overline{G})$. Then for each $s \in S_G$, the 2-packing $\{G_s,M \backslash \overline{G_s}\}$ satisfies $\mathcal{J}_N(\{G_s,M \backslash \overline{G_s}\}) = \mathcal{J}_N(G_s)$, and Theorem \ref{thm:levelset} yields \eqref{eq:firstcheegerneumannls}. 
\end{proof}

\subsection{Higher Cheeger inequalities}    \label{sec:hcheeg} 

On boundaryless manifolds, Miclo \cite{miclo15} and Funano \cite{funano} have proven Cheeger inequalities for $h_{k,\emptyset}$ for all $k \ge 3$. Both papers make use of higher Cheeger inequalities for the graph Laplacian on finite graphs \cite{logt}, following a procedure outlined by Miclo in \cite[Conjecture 13]{miclo08}. 
Miclo states these results for unweighted manifolds, but notes that they also apply to weighted manifolds with $C^\infty$ measures \cite[p.327]{miclo15}. %

\begin{theorem}[{\cite[Theorem 7]{miclo15}}] \label{thm:miclocheeger}
There is a universal constant $\hat{\eta} > 0$ such that, for any boundaryless weighted manifold $(M,g,\mu)$ and for all $k \ge 1$, 
\begin{align}
    \lambda_{k,\emptyset} \le -\frac{\hat{\eta}}{k^6} h_{k,\emptyset}^2. \label{eq:miclocheeger}
\end{align}
\end{theorem}
\begin{theorem}[{\cite[Theorem 13]{miclo15}}] \label{thm:miclorcheeger}
There is a universal constant $\eta$ such that, for any boundaryless weighted manifold $(M,g,\mu)$ and for all $k \ge 1$, 
\begin{align}
    \lambda_{2k,\emptyset} \le -\frac{\eta}{\log(k+1)} h_{k,\emptyset}^2. \label{eq:miclorcheeger}
\end{align}
\end{theorem}

The factor of 2 in the $\lambda_{2k,\emptyset}$ in the previous theorem is arbitrary. Indeed, one can obtain the following from Miclo's proof of the previous theorem: there is a universal constant $\tilde{\eta}$ such that, for any boundaryless weighted manifold $(M,g,\mu)$ and for all $k \ge 1$ and $0<\delta<1$, 
\begin{align}
    \lambda_{k,\emptyset} \le -\frac{\tilde{\eta} \delta^6}{\log(k+1)} h_{\lceil (1-\delta) k \rceil,\emptyset}^2. \label{eq:miclorcheegerdelta}
\end{align}
In particular, taking $\delta=\frac{1}{2}$, we have 
\begin{align}
  \lambda_{2k-1,\emptyset} \le -\frac{\tilde{\eta}}{64 \log(k+1)} h_{k,\emptyset}^2. \label{eq:miclorcheegerodd}  
\end{align}

We {are} not aware of a closed-form expression for the constants in \eqref{eq:miclocheeger}--\eqref{eq:miclorcheegerdelta}. {On boundaryless manifolds with nonnegative Ricci curvature, Liu \cite[Theorem 4.1]{liu} and Funano \cite[Theorem 1.7]{funano} have each proposed higher Buser inequalities, showing that $h_{k,\emptyset}$ goes to 0 if and only if $\lambda_{k,\emptyset}$ goes to 0 (on a sequence of such manifolds).} 

Parini \cite[Theorem 5.4]{parini2} notes that the classical proof of the $k=2$ Neumann Cheeger inequality \eqref{eq:neumannstaticcheeger} extends to the $k=2$ Dirichlet case. Parini states his inequality for eigenfunctions of the $p$-Laplacian for $1 < p < \infty$ on subsets of $\mathbb{R}^n$ with Lipschitz boundary, but the same argument applies on weighted manifolds. 

\begin{theorem} \label{thm:secondcheegerd}
Let $(M,g,\mu)$ be a weighted manifold. Then
\begin{align}
    \lambda_{2,D} \le -\frac{1}{4} h_{2,D}^2. \label{eq:seconddirichletcheeger}
\end{align}
\end{theorem}

{Parini's result applies uniformly to all Riemannian manifolds, because the second eigenfunction of $\Delta_\mu$ always has two nodal domains (as an immediate corollary of Theorem \ref{thm:courant}). This uniform bound} does not generalise directly to higher $k$, since the eigenfunctions corresponding to $\lambda_{k,N}$ or $\lambda_{k,D}$ can sometimes have {far fewer than $k$} nodal domains. Indeed, for any boundaryless $n \ge 3$-dimensional manifold and any $k \ge 1$, there is a metric $g$ on $M$ such that the second eigenspace is $k$-dimensional \cite[p.254]{colindeverdiere}, and hence $\lambda_{k+1,\emptyset}=\lambda_{2,\emptyset}$. 

Madafiglio's (unpublished) Honours thesis \cite{madafiglio} provides a generalisation of Theorem \ref{thm:secondcheegerd}. 
Madafiglio observes that if some eigenfunction with eigenvalue $\lambda_{k,D}$ has $r_k \ge 2$ nodal domains, then $\lambda_{k,D}$ gives an upper bound on $h_{r_k,D}${, applying similar reasoning as in the proofs of \cite[Theorem 5.4]{parini2}---and the classical works in Theorem \ref{thm:firstcheeger}---to the nodal domains of higher eigenfunctions}. The Neumann case follows by similar reasoning. Using Theorem \ref{thm:levelset}, we can obtain a constructive version of Madafigilio's result. 

\begin{theorem}[Higher Cheeger inequality] \label{thm:cheeger}
Let $(M,g,\mu)$ be a weighted manifold. For each $k \ge 1$, let $r_k$ denote the number of nodal domains in any Neumann (resp.\ Dirichlet) eigenfunction $u$ of $\Delta_\mu$ with eigenvalue $\lambda{\le} \lambda_{k,N}$ (resp.\ $\lambda {\le} \lambda_{k,D}$).  
\begin{enumerate}
\item We have (\eqref{eq:highercheegerstaticd} due to \cite{madafiglio})
\begin{align}
    \lambda_{k,N} &\le -\frac{1}{4} h_{r_k,N}^2, \label{eq:highercheegerstatic} \\
    \lambda_{k,D} &\le -\frac{1}{4} h_{r_k,D}^2. \label{eq:highercheegerstaticd}
\end{align}
\item Let $u$ be any Neumann (resp.\ Dirichlet) eigenfunction of $\Delta_\mu$ with $r_k$ nodal domains, and let $G^1,\ldots,G^{r_k} \subset M$ denote the nodal domains of $u$. For each $i$ and each $s \in \range(u^2|_{G^i})$, let $G^i_s$ denote the $s$-superlevel set of $u^2$ on $G^i$, and define $S_{G^i}$ as in \eqref{eq:Sdef} or \eqref{eq:Sdefd}. Then {each $S_{G^i}$ has positive Lebesgue measure satisfying \eqref{eq:Smeasure}}, and for each $\{s_1,\ldots,s_{r_k}\} \in S_{G^1} \times \ldots \times S_{G^{r_k}}$, the collection $\mathcal{A}_{r_k}:=\{G_{s_1}^1,\ldots,G_{s_{r_k}}^{r_k}\}$ is a Neumann (resp.\ Dirichlet) $r_k$-packing of $M$ satisfying $\lambda_{k,N} \le -\frac{1}{4} \mathcal{J}_N(\mathcal{A}_{r_k})^2$ (resp.\ $\lambda_{k,D} \le -\frac{1}{4} \mathcal{J}_D(\mathcal{A}_{r_k})^2$). 
\end{enumerate}
\end{theorem}

\begin{proof}
The sets $G_{s_1}^1,\ldots,G_{s_{r_k}}^{r_k}$ for each $\{s_1,\ldots,s_{r_k}\} \in S_{G^1},\ldots,S_{G^{r_k}}$ are pairwise disjoint, since $G^1,\ldots,G^{r_k}$ are pairwise disjoint, and each $G_{s_i}^i \in \mathscr{P}_N(M)$ (resp.\ $G_{s_i}^i \in \mathscr{P}_D(M)$) by the definitions \eqref{eq:Sdef}--\eqref{eq:Sdefd}. Hence $\mathcal{A}_{r_k}:=\{G^1_{s_1},\ldots,G^{r_k}_{s_{r_k}}\}$ is a Neumann $r_k$-packing for $M$ satisfying $\lambda \le -\frac{1}{4} \mathcal{J}_N(\mathcal{A}_{r_k})^2$ (resp.\ a Dirichlet $r_k$-packing for $M$ satisfying $\lambda \le -\frac{1}{4} \mathcal{J}_D(\mathcal{A}_{r_k})^2$), and \eqref{eq:highercheegerstatic} (resp.\ \eqref{eq:highercheegerstaticd}) follows immediately. 
\end{proof}

We can rewrite part 1 of Theorem \ref{thm:cheeger} as follows: for $k\ge 1$, let $\tilde{r}_k$ be the index of a Neumann (resp.\ Dirichlet) eigenfunction of $\Delta_\mu$ with $\ge k$ nodal domains, when the eigenfunctions are ordered by decreasing eigenvalue. Then 
\begin{align}
    \lambda_{\tilde{r}_k,N} \le -\frac{1}{4} h_{k,N}^2
\end{align}
and
\begin{align}
    \lambda_{\tilde{r}_k,D} \le -\frac{1}{4} h_{k,D}^2,
\end{align}
respectively. We can rewrite equations \eqref{eq:dcheegerr}--\eqref{eq:dcheegerrd} similarly. 

Theorem \ref{thm:cheeger} is intended for situations where an eigenfunction of $\Delta_\mu$ has been calculated explicitly, so that the number of nodal domains can be identified.
{The leading eigenfunctions of Laplacians on compact submanifolds of $\mathbb{R}^n$ can be well approximated with numerical schemes underpinned by convergence theory, such as finite element methods, spectral methods, hpFEM, and so on.}
{When the number of nodal domains is known,} Theorem \ref{thm:cheeger} has the twin advantages that it applies to manifolds with boundary, and that the constant in \eqref{eq:highercheegerstatic} is explicit and small. This allows relatively tight bounds on $h_{k,N}$ or $h_{k,D}$ to be computed be computed even for large $k$, particularly when $\tilde{r}_k$ is close to $k$. {Asymptotically for large $k$, Pleijel's nodal domain theorem \cite{berard-meyer,lena,de-ponti-farinelli-violo} implies that $r_k$ cannot exceed a dimension-dependent fraction of $k$, strictly less than $k$, across a wide range of spaces. One can still obtain useful asymptotic results from Theorem \ref{thm:cheeger}, as described in Section \ref{sec:cheeger-examples}.}

\subsection{Creating more nodal domains using linear combinations of eigenfunctions}

Theorem \ref{thm:cheeger} only allows us to obtain one feature from each of the nodal domains of a single eigenfunction of $\Delta_\mu$. Sometimes, there are $l \le k$ features of interest which appear spread among the first $k$ eigenfunctions, but no single eigenfunction has all $l$ features appearing in separate nodal domains. One may be able to extract these $l$ features and obtain a corresponding bound on $h_{l,N}$ or $h_{l,D}$, by applying an operator known as \emph{soft thresholding} to certain linear combinations of the first $k$ eigenfunctions. 
Soft thresholding with parameter $a>0$ is the map $\tau_a:C^0(M) \to C^0(M)$, $\tau_a(f)(p):=\sign(f(p))\max\{|f(p)|-a,0\}$. Soft thresholding does not increase $W^{1,2}$-norm, and is support-decreasing, in the sense that if $f^{-1}(0) \not\in \{\emptyset,M\}$, then $\supp(\tau_a(f)) \subsetneq \supp(f)$. For some manifolds, there are parameters $\alpha:=\{\alpha_{ij}:1 \le i \le l,1 \le j \le k\}$ for which the $l$ linear combinations $f_{i;\alpha}:=\sum_{j=1}^k \alpha_{ij} u_j$, $i=1,\ldots,l$ of the first $k$ (Neumann or Dirichlet) eigenfunctions of $\Delta_\mu$ are $L^2$-close to a collection of $l$ functions with pairwise disjoint supports \cite[Theorem 19]{davies}. When the eigenfunctions can be computed or approximated explicitly, the parameters $\alpha$ can be chosen using an algorithm such as \emph{sparse eigenbasis approximation} \cite{FRS19}, as discussed after the proof of Proposition \ref{thm:seba}. Each $f_{i;\alpha}$ has support covering all of $M$, as a consequence of the unique continuation theorem \cite{aronszajn}%
\footnote{The function $f_{i,\alpha}$ satisfies $|\Delta_\mu f_{i,\alpha}| \le |\lambda_{k,N}| |f_{i,\alpha}|$ or $|\Delta_\mu f_{i,\alpha}|\le |\lambda_{k,D}| |f_{i,\alpha}|$, so the main theorem of \cite{aronszajn} implies that if $f_{i,\alpha}$ cannot be zero in an open neighbourhood unless it is zero everywhere.}, %
but the thresholded functions $\tau_a(f_{1,\alpha}),\ldots,\tau_a(f_{l,\alpha})$ may have pairwise disjoint supports. Increasing $a$ decreases the supports of $\tau_a(f_{1,\alpha}),\ldots,\tau_a(f_{l,\alpha})$, so one chooses $a$ as small as required to achieve pairwise disjoint supports for $\tau_a(f_{1,\alpha}),\ldots,\tau_a(f_{l,\alpha})$. 
In Proposition \ref{thm:seba} below, we give upper bounds on $h_{l,N}$ or $h_{l,D}$, and prove that some of the level sets of $\tau_a(f_{1,\alpha}),\ldots,\tau_a(f_{l,\alpha})$ yield Cheeger $l$-packings whose Cheeger ratios are bounded above, in terms of $\lambda_{k,N}$ or $\lambda_{k,D}$ and the proportion of mass lost in the thresholding step. 
We illustrate Proposition \ref{thm:seba} in example \ref{thm:example}. 

\begin{proposition} \label{thm:seba}
For any weighted manifold $(M,g,\mu)$, let $u_1,\ldots,u_k$ denote the first $k$ Neumann, resp.\ Dirichlet, eigenfunctions of $\Delta_\mu$ on $M$ for $k \ge 1$. For any $1 \le l \le k$ and any $\alpha \in \mathbb{R}^{l\times k}$, define $f_{1,\alpha},\ldots,f_{l,\alpha}$ by $f_{i,\alpha}:=\sum_{j=1}^k \alpha_{ij}u_j$. Suppose that for some $a>0$, the functions $\tau_a(f_{1,\alpha}),\ldots,\tau_a(f_{l,\alpha})$ are nonzero and have pairwise disjoint supports. 
Then each $\tau_a(f_{i,\alpha})$ has a nodal domain $\tilde{G}^i$ such that letting $\tilde{G}^i_s$ for $s \in \range(\tau_a(f_{i,\alpha})^2|_{\tilde{G}^i})$ denote the $s$-superlevel set of $\tau_a(f_{i,\alpha})^2$ on $\tilde{G}_i$, the set 
\begin{align}\tilde{S}_{\tilde{G}^i} := 
\Bigl\{ s \in \range(\tau_a(f_{i,\alpha})^2|_{\tilde{G}^i}) : \tilde{G}^i_s \in \mathscr{P}_N(M), 
\frac{\||\nabla \tau_a(f_{i,\alpha})|\|_{L^2(\tilde{G}^i;\mu)}^2}{\|\tau_a(f_{i,\alpha})\|_{L^2(\tilde{G}^i;\mu)}^2} \ge \frac{1}{4}\mathcal{J}_N(\tilde{G}^i_s)^2 \Bigr\}, 
\end{align}
resp.\ 
\begin{align} 
\tilde{S}_{\tilde{G}^i} := 
\Bigl\{ s \in \range(\tau_a(f_{i,\alpha})^2|_{\tilde{G}^i}) : \tilde{G}^i_s \in \mathscr{P}_D(M),
\frac{\| |\nabla \tau_a(f_{i,\alpha})|\|_{L^2(\tilde{G}^i;\mu)}^2}{\|\tau_a(f_{i,\alpha})\|_{L^2(\tilde{G}^i;\mu)}^2} \ge \frac{1}{4}\mathcal{J}_D(\tilde{G}^i_s)^2 \Bigr\},
\end{align} 
has positive measure satisfying {
\begin{align}
    \Leb(\tilde{S}_{\tilde{G}^i}) \ge \frac {\mleft\| \overline{\tilde{h}_i} - \tilde{\mathrm{h}}_i \mright\|_{L^1\mleft(\range\mleft(\tau_a(f_{i,\alpha})^2|_{\tilde{G}^i}\mright);\tilde{\mathbb{P}}_i\mright)} \|\tau_a(f_{i,\alpha})\|_{L^2(\tilde{G}^i;\mu)}^2}
    {2 \mleft(\overline{\tilde{h}_i} - \inf_{s \in \smash{\range\mleft(\tau_a(f_{i,\alpha})^2|_{\tilde{G}^i}\mright)}} \tilde{\mathrm{h}}_i(s)\mright) \mu(\tilde{G}^i)}, \label{eq:seba-Smeasure}
\end{align}
where $\overline{\tilde{h}_i}$, $\tilde{\mathrm{h}}_i$ and $\tilde{\mathbb{P}}_i$ are defined in the proof below}. Moreover, for each $\{s_1,\ldots,s_l\} \in \tilde{S}_{\tilde{G}^1} \times \ldots \times \tilde{S}_{\tilde{G}^l}$, the collection $\mathcal{A}_l:=\{\tilde{G}_{s_1}^1,\ldots,\tilde{G}_{s_l}^l\}$ is a Neumann $l$-packing for $M$ satisfying 
\begin{align} 
    \lambda_{k,N} 
    \le -\frac{1}{4} \mathcal{J}_N(\mathcal{A}_l)^2 \max_{1 \le j \le l} \frac{\|\tau_a(f_{j,\alpha})\|_{L^2(M;\mu)}^2}{\|f_{j,\alpha}\|_{L^2(M;\mu)}^2}
    \le -\frac{1}{4} h_{l,N}^2 \max_{1 \le j \le l} \frac{\|\tau_a(f_{j,\alpha})\|_{L^2(M;\mu)}^2}{\|f_{j,\alpha}\|_{L^2(M;\mu)}^2}, 
    \label{eq:sebacheeger}
\end{align}
resp.\ a Dirichlet $l$-packing for $M$ satisfying 
\begin{align}
    \lambda_{k,D} 
    \le -\frac{1}{4} \mathcal{J}_D(\mathcal{A}_l)^2 \max_{1 \le j \le l} \frac{\|\tau_a(f_{j,\alpha})\|_{L^2(M;\mu)}^2}{\|f_{j,\alpha}\|_{L^2(M;\mu)}^2} 
    \le -\frac{1}{4} h_{l,D}^2 \max_{1 \le j \le l} \frac{\|\tau_a(f_{j,\alpha})\|_{L^2(M;\mu)}^2}{\|f_{j,\alpha}\|_{L^2(M;\mu)}^2}. \label{eq:sebacheegerd}
\end{align}
\end{proposition}

\begin{figure}[!hbt]
\centering
\begin{subfigure}[t]{0.4\textwidth}
\includegraphics{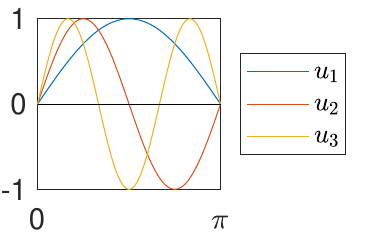}
\caption{Three leading eigenfunctions of $\Delta$ on $[0,\pi]$} \label{fig:interval-orig}
\end{subfigure}
\quad
\begin{subfigure}[t]{0.52\textwidth}
\hspace{1em}\includegraphics{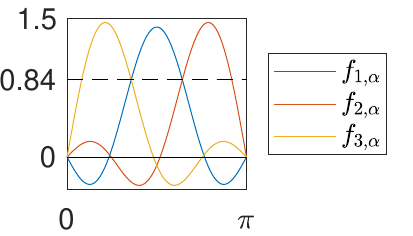}
\caption{Three linear combinations $f_{1,\alpha},f_{2,\alpha},f_{3,\alpha}$ of the eigenfunctions. The regions where each function takes values $>0.84$ are pairwise disjoint.} \label{fig:interval-efuncs}
\end{subfigure}

\hspace*{-1cm}
\begin{subfigure}[t]{0.46\textwidth}
\includegraphics{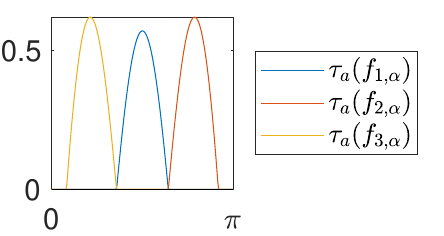}
\caption{The three pairwise disjointly supported functions $\tau_a(f_{1,\alpha}),\tau_a(f_{2,\alpha}),\tau_a(f_{3,\alpha})$, with $a:=0.84$ obtained by soft thresholding.} \label{fig:interval-sparse}
\end{subfigure}%
\nolinebreak\quad%
\begin{subfigure}[t]{0.63\textwidth}
\includegraphics{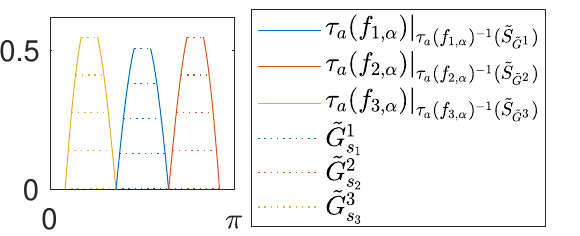}
\caption{Solid lines: restrictions of each $\tau_a(f_{i,\alpha})$ to the set $\tau_a(f_{i,\alpha})^{-1}(\tilde{S}_{\tilde{G}^i})$. Dotted lines: for $i=1,2,3$, each dotted line is a superlevel set $\tilde{G}^i_{s_i}$ of $\tau_a(f_{i,\alpha})$ for five choices of $s_i \in \tilde{S}_{\tilde{G}^i}$. }
\label{fig:interval-packing}
\end{subfigure}

\caption{Eigenfunctions \subref{fig:interval-orig}, linear combinations \subref{fig:interval-efuncs}, thresholded functions \subref{fig:interval-sparse} and superlevel sets for levels in $\tilde{S}_{\tilde{G}^i}$ \subref{fig:interval-packing}, for the first three eigenfunctions of $\Delta$ on $[0,\pi]$.}
\end{figure}

\begin{example} \label{thm:example}
Let $(M,g,\mu)$ denote the interval $[0,\pi]$ equipped with Euclidean distance and Lebesgue measure $\mathrm{Leb}^{{1}}$, and let $u_1,u_2,u_3$ denote the first three Dirichlet eigenfunctions of $\Delta$ on $[0,\pi]$ (shown in Figure \ref{fig:interval-orig}). Using sparse eigenbasis approximation \cite[Algorithm 3.1]{FRS19}, we take $\alpha:=\begin{psmallmatrix}0.77&0&-0.64\\0.45&-0.71&0.54\\0.45&0.71&0.54\end{psmallmatrix}$. Then the linear combinations $f_{1,\alpha}:=\sum_{j=1}^3 \alpha_{ij} u_j$, $j=1,2,3$ of $u_1,u_2,u_3$ (shown in Figure \ref{fig:interval-efuncs}) are $L^2$-close to disjointly supported functions. Applying soft thresholding $\tau_a$ with $a:=0.84$ yields pairwise disjointly supported functions $\tau_a(f_{1,\alpha}),\ldots,\tau_a(f_{3,\alpha})$ (shown in Figure \ref{fig:interval-sparse}). 

 Each $\tau_a(f_{i,\alpha})$ has a single nodal domain $\tilde{G}^i$, and the corresponding positive-measure intervals $\tilde{S}_{\tilde{G}^i}$ are given by $\tilde{S}_{\tilde{G}^1}\approx(0,0.51]$, $\tilde{S}_{\tilde{G}^2}=\tilde{S}_{\tilde{G}^3}\approx (0,0.55]$ (to two decimal places). We show some of the sets $\tilde{G}^i_{s_i}$ for $s_i \in \tilde{S}_{\tilde{G}^i}$, $i=1,2,3$, in Figure \ref{fig:interval-packing}. Proposition \ref{thm:seba} guarantees that each $\tilde{S}_{\tilde{G}^i}$ has positive measure, and that each $\mathcal{A}_3:=\{\tilde{G}^1_{s_1},\tilde{G}^2_{s_2},\tilde{G}^3_{s_3}\}$ for $\{s_1,s_2,s_3\} \in \tilde{S}_{\tilde{G}^1}\times \tilde{S}_{\tilde{G}^2}\times \tilde{S}_{\tilde{G}^3}$ satisfies $\mathcal{J}_D(\mathcal{A}_3) \le 2\sqrt{-\lambda_{3,D}}\frac{\|f_{3,\alpha}\|_{L^2([0,\pi],\mathrm{Leb}^{{1}})}}{\|\tau_a(f_{3,\alpha})\|_{L^2([0,\pi],\mathrm{Leb}^{{1}})}}=7.3$. Some choices for $\{s_1,s_2,s_3\}$ give rise to packings $\mathcal{A}_3$ with Cheeger ratios significantly smaller than this upper bound. Note that for this example, $u_3$ already has 3 nodal domains, so we could use Theorem \ref{thm:cheeger} to obtain a 3-packing instead. 
\end{example}

\begin{proof}
We consider only the Neumann case; the proof of the Dirichlet case is similar. For each $1 \le i \le l$, let $\tilde{G}^i:=\argmin_{\tilde{G}} \frac{\| |\nabla \tau_a(f_{i,\alpha})| \|_{L^2(\tilde{G};\mu)}^2}{\|\tau_a(f_{i,\alpha})\|_{L^2(\tilde{G};\mu)}^2}$, where the minimum is taken over nodal domains $\tilde{G}$ of $\tau_a(f_{i,\alpha})$. 
The level sets of $\tau_a(f_{i,\alpha})$, other than $(\tau_a(f_{i,\alpha}))^{-1}(0)$, are level sets of $f_{i,\alpha} \in C^\infty(M)$, so $\tilde{G}^i_{s_i} \in \mathscr{P}_N(M)$ for almost every $s \in \range(\tau_a(f_{i,\alpha})^2|_{\tilde{G}^i})$ by the reasoning after \eqref{eq:coarea}. 
By applying the reasoning from Theorem \ref{thm:levelset} (\eqref{eq:coarea}--\eqref{eq:rayleigh-h} and after \eqref{eq:cheegerbarbound}) to $\tau_a(f_{i,\alpha})$ on $\tilde{G}^i$, it follows immediately that $\tilde{S}_{\tilde{G}^i}$ has positive measure satisfying \eqref{eq:seba-Smeasure} below, and that $\{\tilde{G}^1_{s_1},\ldots,\tilde{G}^l_{s_l}\} \in \mathscr{P}_{l,N}(M)$ for each $\{s_1,\ldots,s_l\} \in \tilde{S}_{\tilde{G}^1} \times \ldots \times \tilde{S}_{\tilde{G}^l}$. 

We now proceed to prove \eqref{eq:sebacheeger}. Choose any $i \in \{1,\ldots,l\}$. 
Note that for each nodal domain $\tilde{G}$ of $\tau_a(f_{i,\alpha})$, we have $\| |\nabla \tau_a(f_{i,\alpha})| \|_{L^2(\tilde{G};\mu)}^2 \ge \|\tau_a(f_{i,\alpha})\|_{L^2(\tilde{G};\mu)}^2 \frac{\||\nabla \tau_a(f_{i,\alpha})|\|_{L^2(\tilde{G}^i;\mu)}^2}{\|\tau_a(f_{i,\alpha})\|_{L^2(\tilde{G}^i;\mu)}^2}$. Hence 
\begin{align}
    \frac{\||\nabla \tau_a(f_{i,\alpha})|\|_{L^2(M;\mu)}^2}{\|\tau_a(f_{i,\alpha})\|_{L^2(M;\mu)}^2} 
    &= \frac{\sum_{\tilde{G}} \||\nabla \tau_a(f_{i,\alpha})|\|_{L^2(\tilde{G};\mu)}^2}{\sum_{\tilde{G}} \| \tau_a(f_{i,\alpha})\|_{L^2(\tilde{G};\mu)}^2} 
    \ge \frac{\||\nabla \tau_a(f_{i,\alpha})|\|_{L^2(\tilde{G}^i;\mu)}^2}{\|\tau_a(f_{i,\alpha})\|_{L^2(\tilde{G}^i;\mu)}^2}. \label{eq:rayleigh-h-seba}
\end{align}
Recalling that $f_{i,\alpha}=\sum_{j=1}^k \alpha_{ij} u_j$, we have that $\lambda_{k,N} \le -\frac{\||\nabla f_{i,\alpha}|\|_{L^2(M;\mu)}^2}{\|f_{i,\alpha}\|_{L^2(M;\mu)}^2}$ (by e.g.\ \cite[first equation of Proposition 4.5.4]{lablee}, which extends directly to the weighted case). Hence, since $\||\nabla \tau_a(f_{i,\alpha})|\|_{L^2(M)} \le \||\nabla f_{i,\alpha}|\|_{L^2(M)}$, \eqref{eq:rayleigh-h-seba} implies 
\begin{align}
    \lambda_{k,N} 
    \le -\frac{\||\nabla f_{i,\alpha}|\|_{L^2(M;\mu)}^2}{\|f_{i,\alpha}\|_{L^2(M;\mu)}^2}
    &\le -\frac{\||\nabla \tau_a(f_{i,\alpha})|\|_{L^2(M;\mu)}^2}{\|\tau_a(f_{i,\alpha})\|_{L^2(M;\mu)}^2} \frac{\|\tau_a(f_{i,\alpha})\|_{L^2(M;\mu)}^2}{\|f_{i,\alpha}\|_{L^2(M;\mu)}^2} \nonumber \\
    &\le -\frac{\||\nabla \tau_a(f_{i,\alpha})|\|_{L^2(\tilde{G}^i;\mu)}^2}{\|\tau_a(f_{i,\alpha})\|_{L^2(\tilde{G}^i;\mu)}^2} \frac{\| \tau_a(f_{i,\alpha})\|_{L^2(M;\mu)}^2}{\| f_{i,\alpha} \|_{L^2(M;\mu)}^2}. \label{eq:rayleigh-lambda-seba}
\end{align}
Hence, by the definition of $\tilde{S}_{\tilde{G}^i}$ in the proposition statement, for each $s_i \in \tilde{S}_{\tilde{G}^i}$, we have 
\begin{align*}
    \lambda_{k,N} 
    \le -\frac{1}{4} \mathcal{J}_N(\tilde{G}^i_{s_i})^2 \frac{\|\tau_a(f_{i,\alpha})\|_{L^2(M;\mu)}^2}{\|f_{i,\alpha}\|_{L^2(M;\mu)}^2}. 
\end{align*}
Applying this reasoning for each $i \in \{1,\ldots,k\}$ and recalling that $\tau_a(f_{1,\alpha}),\ldots,\tau_a(f_{l,\alpha})$ have pairwise disjoint supports yields \eqref{eq:sebacheeger}. 

Lastly, we state our lower bound on the measure of each $\tilde{S}_{\tilde{G}^i}$. 
Similarly to Theorem \ref{thm:levelset}, we define $\overline{\tilde{h}_i}:=\frac{1}{\|\tau_a(f_{i,\alpha})\|_{L^2(\tilde{G}^i)^2}} \int_{\range\mleft(\tau_a(f_{i,\alpha})^2|_{\tilde{G}^i}\mright)} \mathcal{J}_N(\tilde{G}_s^i) \mu(\tilde{G}_s^i) \dd s$ and $\tilde{\mathrm{h}}_i(s):=\mathcal{J}_N(\tilde{G}_s^i)$, and we define the probability measure $\tilde{\mathbb{P}}_i$ on $\range(\tau_a(f_{i,\alpha})^2|_{\tilde{G}^i}) \subset \mathbb{R}$ by $\tilde{\mathbb{P}}_i(L):=\int_L \frac{\mu(\tilde{G}_s^i)}{\|\tau_a(f_{i,\alpha})\|_{L^2(\tilde{G}^i;\mu)}^2} \dd s$. Then the reasoning for \eqref{eq:Smeasure} implies {\eqref{eq:seba-Smeasure}}.
A similar result applies in the Dirichlet case, replacing $\mathcal{J}_N$ with $\mathcal{J}_D$ in the definitions of $\overline{\tilde{h}_i},\tilde{\mathrm{h}}_i,\tilde{\mathbb{P}}_i$. 
\end{proof}

In numerical calculations, the $\alpha_{ij}$ in Proposition \ref{thm:seba} can be readily computed using the sparse eigenbasis approximation algorithm of \cite[Algorithm 3.1]{FRS19}. The orthogonal matrix $R$ produced by that algorithm can be used as the matrix $\alpha$. The resulting $f_{1,\alpha},\ldots,f_{k,\alpha}$ form an orthogonal basis for $\Span\{u_1,\ldots,u_k\}$, such that for some fixed $a>0$, each $\tau_a(f_{i,\alpha})$ (defined before Proposition \ref{thm:seba}) is \emph{sparse}, i.e.\ each $\supp(\tau_a(f_{i,\alpha}))$ is small. Using a larger $a' \ge a$ will create further support reductions.


\subsection{The dynamic Cheeger inequalities} \label{sec:dcheeg}

\subsubsection{Preliminaries on higher dynamic Cheeger constants and dynamic Laplacian eigenvalues} \label{sec:introdynamic}
We can generalise Theorems \ref{thm:miclocheeger}--\ref{thm:cheeger} into the setting of non-autonomous, advective dynamical systems. Many fluidic and geophysical flows can be modeled using purely advective dynamics. Such flows can be represented as a collection of time-indexed diffeomorphisms acting on an initial-time manifold, where each diffeomorphism sends a point in the initial-time manifold to its position at the corresponding future time. These diffeomorphisms are physically meaningful, because they describe the fluid motion and evolve subsets of the initial-time manifold according to this motion.

The global behaviour of many fluidic and geophysical flows can be understood by separating the phase space (the physical space containing the fluid) into \emph{coherent sets} \cite{F15}, i.e.\ regions that are ``as dynamically disconnected as possible'' \cite{FJ18}. One approach in purely advective, finite-time nonautonomous systems is to identify subsets of the phase space whose boundary measures remain small over time, relative to the measures of those subsets. These volume ratios are known as \emph{dynamic Cheeger ratios} \cite{F15,FKw17,FJ18}, and sets which locally minimise this ratio are known as \emph{coherent sets}. The infima of these ratios are known as the \emph{dynamic Cheeger constants} \cite{F15,FKw17,FJ18}. 
The dynamic Cheeger constants generalise the (static) Cheeger constants of Definition \ref{def:whcheeg}. 

Calculating a dynamic Cheeger constant exactly is generally impractical. Instead, approximate coherent sets can be obtained from the eigenfunctions of a specific weighted Laplace-Beltrami operator called the \emph{dynamic Laplacian}. There are existing upper bounds on the first non-zero dynamic Cheeger constant in terms of the first non-zero eigenvalue of the dynamic Laplacian \cite{F15,FKw17,FJ18}. 

In practice, the higher eigenfunctions of the dynamic Laplacian reveal additional coherent sets (see e.g.\ \cite{FRS19}). Below, we introduce higher dynamic Cheeger constants, analogous to the (static) higher Cheeger constants of Definition \ref{def:whcheeg}, to quantify these additional coherent sets. We show that the higher dynamic Cheeger constants are bounded above by the eigenvalues of {the dynamic Laplacian} (Theorems \ref{thm:dmiclo}, \ref{thm:dmiclor} and \ref{thm:dcheeger}), and in particular that the eigenfunctions of {the dynamic Laplacian} reveal coherent sets whose dynamic Cheeger ratios are bounded above (Theorems \ref{thm:dcheeger} and \ref{thm:dseba}). 

\begin{definition} \label{def:dynsys}
A \emph{dynamical system} $\mathcal{T}:=(\mathrm{T}, \{(M_t,g_t)\}_{t \in \mathrm{T}}, \{\Phi^{(t)}\}_{t \in \mathrm{T}})$ or $\mathcal{T}:=(\mathrm{T}, \allowbreak \{(M_t,g_t,\mu_t)\}_{t \in \mathrm{T}}, \allowbreak \{\Phi^{(t)}\}_{t \in \mathrm{T}})$ consists of the following:
\begin{itemize}
    \item A time index set $\mathrm{T}:=\{0,1,\ldots, t_{\max}\}$. 
    \item A time-indexed family of Riemannian manifolds $\{(M_t,g_t)\}_{t \in \mathrm{T}}$ or weighted manifolds $\{(M_t,g_t,\mu_t)\}_{t \in \mathrm{T}}$, where in the unweighted case, for $t \in \mathrm{T}$ we take $\mu_t$ to denote Riemannian volume on $M_t$.  
    \item A time-indexed family of $C^\infty$ diffeomorphisms $\{\Phi^{(t)}\}_{t \in \mathrm{T}}$, which are \emph{measure-preserving} in the sense $\mu_t = \mu_0 \circ (\Phi^{(t)})^{-1}$ (we call such $\Phi^{(t)}$ \emph{volume-preserving} if each $\mu_t$ is Riemannian volume). 
\end{itemize}
\end{definition}

We use the following notation. Since $\Phi^{(t)}$ for $t \in \mathrm{T}$ is a measure-preserving diffeomorphism, the \emph{pushforward} $\Phi^{(t)}_*:C^\infty(M_0) \to C^\infty(M_t)$ is given by $\Phi^{(t)}_* f := f \circ (\Phi^{(t)})^{-1}$, and the \emph{pullback} $(\Phi^{(t)})^*:C^\infty(M_t) \to C^\infty(M_0)$ is given by $(\Phi^{(t)})^* f := f \circ \Phi^{(t)}$. We also define the pullback Riemannian metric $(\Phi^{(t)})^*g_t$ given by $(\Phi^{(t)})^*g_t:= g_t(\ddd\Phi^{(t)} \,\cdot \,, \ddd\Phi^{(t)} \,\cdot \,)$, where $\ddd\Phi^{(t)}$ is the differential of $\Phi^{(t)}$ (see e.g.\ \cite[p.55]{lee}). For $t \in \mathrm{T}$, we let $(\mu_t)_{n-1}$ 
denote the $n-1$-dimensional Hausdorff measure on $M_t$ constructed from $\mu_t$ and $g_t$. For $s, s+t \in \mathrm{T}$, we write $\Phi_s^{(t)}:=\Phi^{(s+t)} \circ (\Phi^{(s)})^{-1}$. 

We define the higher dynamic Cheeger constants as follows. 
\begin{definition}[Higher dynamic Cheeger constants] \label{def:dcheeg}
Consider a dynamical system $\mathcal{T}$. For $k \ge 1$, the \emph{dynamic Neumann Cheeger ratio} of a $k$-packing $\{A_1,\ldots,A_k\} \in \mathscr{P}_{k,D}(M_0)$ is 
\begin{align}
    \mathcal{J}_N^d(\{A_1,\ldots,A_k\})&:=\max_{1 \le i \le k} \frac{\sum_{t=0}^{t_{\max}} (\mu_t)_{n-1}(\Phi^{(t)}(\partial^{M_0} A_i))}{|\mathrm{T}| \mu_0(A_i)}. \label{eq:dcheegerration} 
\end{align}
The \emph{dynamic Dirichlet Cheeger ratio} of a Dirichlet $k$-packing $\{A_1,\ldots,A_k\} \in \mathscr{P}_{k,D}(M_0)$ is 
\begin{align}
    \mathcal{J}_D^d(\{A_1,\ldots,A_k\})&:=\max_{1 \le i \le k} \frac{\sum_{t=0}^{t_{\max}} (\mu_t)_{n-1}(\Phi^{(t)}(\partial A_i))}{|\mathrm{T}| \mu_0(A_i)}. \label{eq:dcheegerratiod} 
\end{align}
The \emph{$k$th dynamic Neumann} and \emph{dynamic Dirichlet Cheeger constants} for $\mathcal{T}$ are
\begin{align}
    h_{k,N}^d&:=\inf_{\{A_1,\ldots,A_k\} \in \mathscr{P}_k(M_0)} \mathcal{J}_N^d(\{A_1,\ldots,A_k\}) \label{eq:dcheegn} \\
    h_{k,D}^d&:=\inf_{\{A_1,\ldots,A_k\} \in \mathscr{P}_k(M_0)} \mathcal{J}_D^d(\{A_1,\ldots,A_k\}). \label{eq:dcheegd}
\end{align}
For $A \in \mathscr{P}_N(M_0)$, resp.\ $A \in \mathscr{P}_D(M_0)$, we will occasionally write $\mathcal{J}_N^d(A)$ instead of $\mathcal{J}_N^d(\{A\})$, resp.\ $\mathcal{J}_D^d(A)$ instead of $\mathcal{J}_D^d(\{A\})$, for convenience. 
\end{definition}

The Neumann dynamic Cheeger constant $h_{2,N}^d$ was originally defined requiring $A_1$ and $A_2$ to partition $M_0$ \cite{F15}, whereas $\eqref{eq:dcheegn}$ only requires them to form a packing of $M_0$. This does not change the value of $h_{2,N}^d$, by the reasoning after definition \ref{def:whcheeg}. Note that since the $\Phi^{(t)}$ are measure-preserving, we have $|\mathrm{T}|\mu_0(A_i)=\sum_{t=0}^{t_{\max}} \mu_t(\Phi^{(t)}(A_i))$, i.e.\ the denominators in \eqref{eq:dcheegerration}-\eqref{eq:dcheegerratiod} are $|\mathrm{T}|$ times the time averages of the measures of the $A_i$.  

When considering dynamical systems, we let $\Delta_{g_t,\mu_t}$ denote the weighted Laplace-Beltrami operator on $(M_t,g_t,\mu_t)$. The dynamic Laplacian \cite{F15,FKw17} is 
\begin{align}
    \Delta^d:=\frac{1}{|\mathrm{T}|} \sum_{t=0}^{t_{\max}} (\Phi^{(t)})^* \Delta_{g_t,\mu_t} \Phi_*^{(t)}. \label{eq:defweighteddynamiclaplacian}
\end{align}
We consider Dirichlet and dynamic Neumann eigenproblems for $\Delta^d$. The dynamic Neumann eigenproblem is to find $u \in C^\infty(M_0)$ and $\lambda \in \mathbb{R}$, such that 
\begin{align}
    \Delta^d u = \lambda u, \label{eq:dynlaplacian}
\end{align}
subject to the \emph{dynamic Neumann boundary condition} (if $\partial M_0 \not= \emptyset$)  
\begin{align}
    \frac{1}{|\mathrm{T}|} \sum_{t=0}^{t_{\max}} \frac{\partial }{\partial \mathbf{n}_t}\left( (\Phi^{(t)})_* u \right) &= 0 \quad \text{on } \partial M_0, \label{eq:dneumannbc}
\end{align}
where $\mathbf{n}_t$ denotes an outward unit normal vector to $\partial M_t$ \cite[Theorem 4.1]{F15} \cite[Theorem 4.4]{FKw17}. Dynamic Neumann boundary conditions are the natural boundary condition as discussed in \cite[pp.9--10]{F15} and \cite[p.16]{FKw17}. 
There is an orthogonal Schauder basis for $L^2(M_0,\mu_0)$ consisting of eigenfunctions for \eqref{eq:dynlaplacian} satisfying \eqref{eq:dneumannbc} \cite[Theorem 4.4]{FKw17}. 
The corresponding eigenvalues form a non-positive decreasing sequence accumulating only at $-\infty$, and we denote them $0 = \lambda_{1,N}^d > \lambda_{2,N}^d \ge \lambda_{3,N}^d \ge \ldots$. 

The Dirichlet eigenproblem is to find $u \in C^\infty(M_0)$ and $\lambda \in \mathbb{R}$ satisfying \eqref{eq:dynlaplacian}, subject to
\begin{align}
    u=0 \quad \text{on } \partial M_0. \label{eq:ddirichletbc}
\end{align}
By standard variational arguments as in e.g.\ \cite[Theorem 4.3.1]{lablee} and elliptic regularity theorems as in \cite[Theorem 8.14]{gilbarg}, there is an orthogonal Schauder basis for $L^2(M_0,\mu_0)$ of $C^\infty(M_0)$ eigenfunctions for \eqref{eq:dynlaplacian} satisfying \eqref{eq:ddirichletbc}. The corresponding eigenvalues form a negative decreasing sequence accumulating only at $-\infty$, and we denote them $0>\lambda_{1,D}^d > \lambda_{2,D}^d \ge \lambda_{3,D}^d \ge \ldots$. 

We have the following variational formula for the eigenvalues, in the dynamic Neumann setting \cite{FKw17}. 
\begin{proposition} \label{thm:dvariational}
Let $\mathcal{T}$ be a dynamical system, and let $u_1^d,u_2^d,\ldots$ denote a complete orthogonal basis of dynamic Neumann eigenfunctions of $\Delta^d$ corresponding to $\lambda_{1,N}^d,\lambda_{2,N}^d,\ldots$ (resp.\ $\lambda_{1,D}^d,\lambda_{2,D}^d,\ldots$). Then for each $k \ge 1$, we have
\begin{align}
    \lambda_{k,N}^d &= -\inf_{\substack{f \in W^{1,2}(M_0) \\ \int_{M_0} u_i^d f \dd\mu_0=0, \forall i \in \{1,\ldots,k-1\}}} \frac{\sum_{t=0}^{t_{\max}} \| |\nabla_{g_t} \Phi^{(t)}_* f| \|_{L^2(M_i;\mu_i)}^2}{|\mathrm{T}|\| f \|_{L^2(M_0;\mu_0)}^2}, \label{eq:dvariationaln} 
\end{align}
and the infimum is attained when $f$ is a dynamic Neumann eigenfunction of $\Delta^d$ with eigenvalue $\lambda_{k,N}^d$. 
\end{proposition}
Extending the reasoning in e.g.\ \cite[pp.16--17]{chavel84} to the dynamic case yields that the infimum in \eqref{eq:dvariationaln} is attained if and only if $f$ is a dynamic Neumann eigenfunction of $\Delta^d$ with eigenvalue $\lambda_{k,N}^d$. 
This proposition also extends directly to the Dirichlet case, by similar arguments. Let $u_1^d,u_2^d,\ldots$ denote a complete orthogonal basis of Dirichlet eigenfunctions of $\Delta^d$ corresponding to $\lambda_{1,N}^d,\lambda_{2,N}^d,\ldots$ (resp.\ $\lambda_{1,D}^d,\lambda_{2,D}^d,\ldots$). Then for each $k \ge 1$, we have
\begin{align}
    \lambda_{k,D}^d &= -\inf_{\substack{f \in W^{1,2}_0(M_0) \\ \int_{M_0} u_i^d f \dd\mu_0=0, \forall i \in \{1,\ldots,k-1\}}} \frac{\sum_{t=0}^{t_{\max}} \| |\nabla_{g_t} \Phi^{(t)}_* f| \|_{L^2(M_t;\mu_t)}^2}{|\mathrm{T}|\| f \|_{L^2(M_0;\mu_0)}^2}, \label{eq:dvariationald}
\end{align}
and the infimum is attained if and only if $f$ is a Dirichlet eigenfunction of $\Delta^d$ with eigenvalue $\lambda_{k,D}^d$. 
Since $\Delta^d$ is an elliptic operator, Courant's nodal domain theorem (Theorem \ref{thm:courant}) extends to the eigenfunctions of $\Delta^d$. 

\begin{corollary}[to Theorem \ref{thm:courant}] \label{thm:dcourant}
For any dynamical system $\mathcal{T}$, the $k$th dynamic Neumann (resp.\ Dirichlet) eigenfunction $u_k$ of $\Delta^d$ has at most $k$ nodal domains. 
\end{corollary}

\begin{proof}
The proof is the same as that for Theorem \ref{thm:courant}, replacing $M$, $\mu$ and $\lambda_{k,N}$ with $M_0$, $\mu_0$ and $\lambda_{k,N}^d$, replacing the Rayleigh quotients $\frac{\||\nabla \cdot|\|_{L^2(\cdot;\mu)}^2}{\|\cdot\|_{L^2(\cdot;\mu)}^2}$ as in Theorem \ref{thm:variational} with dynamic Rayleigh quotients $\frac{\sum_{t=0}^{t_{\max}} \||\nabla_{g_t} \Phi^{(t)}_* \cdot |\|_{L^2(\cdot;\mu_0)}^2}{|\mathrm{T}|\|\cdot\|_{L^2(\cdot;\mu_0)}^2}$ as in Proposition \ref{thm:dvariational}, and replacing \eqref{eq:nablau}--\eqref{eq:lambdauu} with \eqref{eq:dgreen} and the reasoning used to obtain \eqref{eq:dgreenlambda}. 
\end{proof}

The operator $\Delta^d$ can be expressed as the weighted Laplace-Beltrami operator $\Delta_{\bar{g},\mu_0}$ on $(M_0,\bar{g},\mu_0)$, where $\bar{g}$ (called the \emph{geometry of mixing metric} \cite{karrasch}) is  the `harmonic mean'\footnote{%
$\bar{g}$ is defined via the \emph{inverse metric}. The inverse metric of a Riemannian metric $g$ on $M_0$ is given by $g^{-1}:T^* M_0 \times T^* M_0 \to \mathbb{R}$, $g^{-1}(\eta,\omega):=g(\eta^\sharp,\omega^\sharp)$, where $\sharp$ denotes raising an index (see e.g.\ \cite[p.342]{lee}). Then $\bar{g}$ is the unique metric on $M_0$ for which $\bar{g}^{-1}(\eta,\omega):=\frac{1}{|\mathrm{T}|} \sum_{t=0}^{t_{\max}} ((\Phi^{(t)})^* g_t)^{-1}(\eta,\omega)$.} 
of the pullbacks $(\Phi^{(t)})^* g_t$ of the metrics $g_t$ to the initial-time manifold $M_0$ \cite{karrasch}.  
Note that even if each $\mu_t$ is Riemannian volume on $(M_t,g_t)$, $\mu_0$ is not necessarily Riemannian volume on $(M_0,\bar{g})$ \cite[section 4.1.3]{karrasch}. 
\begin{proposition}[{\cite[pp.1864, 1875]{karrasch}}] \label{thm:geommixing}
In any dynamical system, $\Delta^d$ is the weighted Laplace-Beltrami operator for the Riemannian manifold $(M_0,\bar{g},\mu_0)$, i.e. 
\begin{align}
    \Delta^d = \Delta_{\bar{g}, \mu_0}. \label{eq:geommixing}
\end{align}
\end{proposition}

For any dynamical system $\mathcal{T}$, let $\nabla_{g_t}$ and $\nabla_{\bar{g}}$ denote the gradient operator for the time-$t$ manifold $(M_t,g_t,\mu_t)$ and the geometry of mixing manifold $(M_0,\bar{g},\mu_0)$, respectively. It follows immediately from the definition of $\bar{g}$ that $|\nabla_{\bar{g}} f|^2=\frac{1}{|\mathrm{T}|} \sum_{t=0}^{t_{\max}} \bigl|\nabla_{g_t} \Phi^{(t)}_* f\bigr|^2$ for $f \in W^{1,2}(M_0)$. The Neumann boundary condition for the geometry of mixing manifold is the same as the dynamic Neumann boundary condition \cite[p.1864]{karrasch}. For $A \in \mathscr{P}_N(M_0)$ or $A \in \mathscr{P}_D(M_0)$, respectively, we denote the (Neumann or Dirichlet) Cheeger ratio of $A$ on the geometry of mixing manifold by $\mathcal{J}_N(A;\bar{g},\mu_0)$ or $\mathcal{J}_D(A;\bar{g},\mu_0)$, respectively. Then $\mathcal{J}_N(\cdot;\bar{g},\mu)$ and $\mathcal{J}_D(\cdot;\bar{g},\mu)$ give upper bounds on the dynamic Cheeger ratios and dynamic Cheeger constants \cite[Proposition 4.3]{karrasch-schilling}:
\begin{align}
\mathcal{J}_N^d(A) &\le \mathcal{J}_N(A;\bar{g},\mu), \quad \forall A \in \mathscr{P}_N(M_0) \label{eq:schillingn} \\
\mathcal{J}_D^d(A) &\le \mathcal{J}_D(A;\bar{g},\mu), \quad \forall A \in \mathscr{P}_D(M_0). \label{eq:schillingd} 
\end{align}

The bounds in Theorem \ref{thm:firstcheeger} have been extended to the dynamic setting. 
\begin{theorem}[Dynamic Cheeger inequality {\cite{F15,FKw17,FJ18}}] \label{thm:firstdynamiccheeger} \hfill \\
\begin{itemize}
\item \cite[Theorem 3.2]{F15}, \cite[Theorem 4.5]{FKw17}: For any dynamical system, we have
\begin{align}
    \lambda_{2,N}^d \le -\frac{1}{4} (h_{2,N}^d)^2. \label{eq:dynfirstcheeger}
\end{align}
\item \cite[Theorem 2]{FJ18} For any dynamical system such that each $(M_t,g_t,\mu_t)$ is an $n$-dimensional, $C^\infty$ submanifold of $\mathbb{R}^n$ equipped with the Euclidean metric and Lebesgue measure, we have 
\begin{align}
    \lambda_{1,D}^d \le -\frac{1}{4} (h_{1,D}^d)^2.
\end{align}
\end{itemize}
\end{theorem}
Combining the approach from \cite{FKw17} and \cite{FJ18}, equation \eqref{eq:dynfirstcheeger} extends to dynamical systems on arbitrary weighted Riemannian manifolds as in Definition \ref{def:dynsys}.

Similarly to the static case, we can give constructive versions of the dynamic Cheeger inequality (Theorem \ref{thm:dlevelset} and Corollary \ref{thm:dfirstcheegerd}). Specifically, we show that within any nodal domain of an eigenfunction $u$ of $\Delta^d$, a positive-measure collection of superlevel sets of $u$ have their dynamic Cheeger ratio bounded above by the corresponding eigenvalue (Theorem \ref{thm:dlevelset}). This immediately yields a constructive version of Theorem \ref{thm:firstdynamiccheeger} (Corollary \ref{thm:dfirstcheegerd}). 

%
\begin{theorem} \label{thm:dlevelset}
Let $\mathcal{T}$ be a dynamical system, and let $u$ be some Neumann, resp.\ Dirichlet, eigenfunction of $\Delta^d$ with eigenvalue $\lambda$. Let $G \subset M_0$ be any nodal domain of $u$. Then, defining 
\begin{align}
    G_s:=\{p \in G : u(p)^2>s\} \label{eq:dGsdef},
\end{align}
the set 
\begin{align}
    S_G:=\mleft\{s \in \range(u^2|_G) : G_s \in \mathscr{P}_N(M_0), \lambda \le -\frac{1}{4}\mathcal{J}_N^d(G_s)^2 \mright\}, \label{eq:dSdef} 
\end{align}
resp.\  
\begin{align}
    S_G:=\mleft\{s \in \range(u^2|_G) : G_s \in \mathscr{P}_D(M_0),\lambda \le -\frac{1}{4} \mathcal{J}_D^d(G_s)^2 \mright\}, \label{eq:dSdefd} 
\end{align} 
has positive Lebesgue measure satisfying the lower bound {
\begin{align}
    \Leb(S_G) \ge \frac{\|{\bar{h}^d}-\mathrm{h}^d\|_{L^1(\range(u^2|_G);\mathbb{P}^d)} \|u\|_{L^2(G;\mu_0)}^2}{2({\bar{h}^d} - \inf_{s \in \range(u^2|_G)} \mathrm{h}^d(s))\mu_0(G)}, \label{eq:dSmeasure}
\end{align}
where $\mathrm{h}^d$, $\bar{h}^d$ and $\mathbb{P}^d$ are defined in the proof below.}
\end{theorem}

\begin{proof}
The proof proceeds as for Theorem \ref{thm:levelset}. 
For each $t \in \mathrm{T}$, define $\phi_t \in C^\infty(M_t)$ via $\ddd\mu_t=\e^{\phi_t} \dd V$, and observe that $\range((\Phi^{(t)}_* u)^2|_{\Phi^{(t)}(G)})=\range(u^2|_G)$ and that for each $s \in \range(u^2|_G)$, $\Phi^{(t)}(G_s)$ is the superlevel set of $\Phi^{(t)}_* u$ on $\Phi^{(t)}(G)$. Replacing $(M,g,\mu)$, $\phi$, $G$ and $u$, respectively, with $(M_t,g_t,\mu_t)$, $\phi_t$, $\Phi^{(t)}(G)$ and $\Phi^{(t)}_* u$, respectively, in each of \eqref{eq:hdef}, \eqref{eq:rayleigh-h} and \eqref{eq:nablau}--\eqref{eq:nodalrayleigh} yields 
\begin{align}
    \bar{h}:={}&\frac{\int_{\range(u^2|_G)} \mathcal{J}_N(\Phi^{(t)}(G_s)) \mu_t (\Phi^{(t)}(G_s)) \dd s}{\|\Phi^{(t)}_* u\|_{L^2(\Phi^{(t)}(G);\mu_t)}^2}, \label{eq:dhdef} \\
    \frac{\||\nabla_{g_t} \Phi^{(t)}_* u|\|_{L^2(\Phi^{(t)}(G);\mu_t)}^2}{\|\Phi^{(t)}_* u\|_{L^2(\Phi^{(t)}(G);\mu_t)}^2}
    \ge {}&\frac{1}{4} \bar{h}^2, \label{eq:drayleigh} \\
    \||\nabla_{g_t} \Phi^{(t)}_* u|\|_{L^2(\Phi^{(t)}(G);\mu_t)}^2
    = {}&\int_{\Phi^{(t)}(G)} \nabla_{g_t} \Phi^{(t)}_* u \cdot (\e^{\phi_t} \nabla_{g_t} \Phi^{(t)}_* u) \dd V \nonumber \\
    = {}&-\int_{\Phi^{(t)}(G)} u \cdot (\Delta_{g_t,\mu_t} \circ \Phi^{(t)}_*) u \dd\mu_t{.} \label{eq:dgreen} 
\end{align}
Multiplying \eqref{eq:drayleigh} by $\|\Phi^{(t)}_* u\|_{L^2(\Phi^{(t)}(G);\mu_t)}^2$, replacing $\||\nabla_{g_t} \Phi^{(t)}_* u|\|_{L^2(\Phi^{(t)}(G);\mu_t)}^2$ with the right-hand side of \eqref{eq:dgreen}, and then replacing $\bar{h}$ with its definition \eqref{eq:dhdef}, yields 
\begin{align*}
    -\int_{\Phi^{(t)}(G)} \Phi^{(t)}_* u \cdot \bigl(\Delta_{g_t,\mu_t} \circ \Phi^{(t)}_* \bigr) u \dd \mu_t 
    &\ge \frac{1}{4} \bar{h}^2 \|\Phi^{(t)}_* u\|_{L^2(\Phi^{(t)}(G);\mu_t)}^2 \\
    &= \frac{\mleft(\int_{\range(u^2|_G)} \mathcal{J}_N (\Phi^{(t)}(G_s)) \mu_t(\Phi^{(t)}(G_s)) \dd s\mright)^2}{4\|\Phi^{(t)}_* u\|_{L^2(\Phi^{(t)}(G);\mu_t)}^2}.
\end{align*}
Since $\Phi^{(t)}$ is measure-preserving, this is equivalent to 
\begin{align}
    -\int_G u \cdot \bigl( (\Phi^{(t)})^* \circ \Delta_{g_t,\mu_t} \circ \Phi^{(t)}_* \bigr) u \dd\mu_0 
    &\ge \frac{\mleft(\int_{\range(u^2|_G)} \mathcal{J}_N(\Phi^{(t)}(G_s)) \mu_0(G_s) \dd s \mright)^2}{4\|u\|_{L^2(G;\mu_0)}^2}. \label{eq:dlaplaceh}
\end{align}
Now, definition \eqref{eq:defweighteddynamiclaplacian} and our choice of $u$ imply $\frac{1}{|\mathrm{T}|}\sum_{t=0}^{t_{\max}} \bigl( (\Phi^{(t)})^* \circ \Delta_{g_t,\mu_t} \circ \Phi^{(t)}_* \bigr) u=\Delta^d u=\lambda u$, so summing \eqref{eq:dlaplaceh} over $t$ and dividing by $-|\mathrm{T}|\|u\|_{L^2(G;\mu_0)}^2$ {yield}
\begin{align}
    \lambda \le -\frac{1}{4|\mathrm{T}| \|u\|_{L^2(G;\mu_0)}^4} \sum_{t=0}^{t_{\max}} \biggl(\int_{\range(u^2|_G)}\mathcal{J}_N(\Phi^{(t)}(G_s)) \mu_0(G_s) \dd s \biggr)^2. \label{eq:dgreenlambda}
\end{align}
Using the relation $-\sum_{t=0}^{t_{\max}} x_t^2 \le -\frac{1}{|\mathrm{T}|} \mleft(\sum_{t=0}^{t_{\max}} x_t\mright)^2$ for $x \in \mathbb{R}^{|\mathrm{T}|}$, this bound becomes 
\begin{align}
    \lambda &\le -\frac{1}{4|\mathrm{T}|^2 \|u\|_{L^2(G;\mu_0)}^4} \mleft( \sum_{t=0}^{t_{\max}} \int_{\range(u^2|_G)} \mathcal{J}_N(\Phi^{(t)}(G_s)) \mu_0(G_s) \dd s \mright)^2 = -\frac{1}{4} ({\bar{h}^d})^2, \label{eq:drayleigh-h}
\end{align} 
where ${\bar{h}^d}:=\frac{1}{\|u\|_{L^2(G;\mu_0)}^2} \int_{\range(u^2|_G)} \mathcal{J}_N^d (G_s) \mu_0(G_s) \dd s$. 
Thus, by the reasoning after \eqref{eq:cheegerbarbound}, the set $S_G$ defined in \eqref{eq:dSdef} has positive measure. 
{Defining} the probability measure $\mathbb{P}^{{d}}$ on $\range(u^2|_G)$ by $\mathbb{P}^{{d}}(L):=\int_L \frac{\mu_0(G_s)}{\|u\|_{L^2(G;\mu_0)}^2} \dd s$, and {letting} $\mathrm{h}^d(s):=\mathcal{J}_N^d(G_s)${,} the reasoning for \eqref{eq:Smeasure} yields {\eqref{eq:dSmeasure}}.
\end{proof}

\begin{corollary} \label{thm:dfirstcheegerd}
For any dynamical system $\mathcal{T}$, and for any dynamic Neumann eigenfunction $u$ of $\Delta^d$ corresponding to $\lambda_{2,N}^d$, there is a nodal domain $G$ of $u$ such that the set $S_G$ defined in \eqref{eq:dSdef} has positive measure {and satisfies \eqref{eq:dSmeasure}}, and for $s \in S_G$, defining $G_s$ as in \eqref{eq:dGsdef}, the 2-packing $\{G_s,M \backslash \overline{G_s}\}$ satisfies 
\begin{align}
    \lambda_{2,N}^d \le -\frac{1}{4} \mathcal{J}_N^d(\{G_s,M \backslash \overline{G_s}\})^2. \label{eq:dfirstcheegern}
\end{align}
If $\partial M_0 \ne \emptyset$, the leading Dirichlet eigenfunction $\lambda_{1,D}^d$ of $\Delta^d$ is simple, and the corresponding eigenfunction $u$ has only a single nodal domain $G=M_0 \backslash \partial M_0$. The set $S_G$ defined in \eqref{eq:dSdefd} has positive measure, and for $s \in S_G$, the set $G_s$ defined in \eqref{eq:Gsdef} satisfies 
\begin{align}
    \lambda_{1,D}^d \le -\frac{1}{4} \mathcal{J}_D^d(G_s)^2. \label{eq:dfirstcheegerd}
\end{align}
\end{corollary}
\begin{proof}
In the Dirichlet case, we mostly follow the proof of \cite[Proposition 4.5.8]{lablee}. Corollary \ref{thm:dcourant} ensures that any Dirichlet eigenfunction $u$ of $\Delta^d$ corresponding to $\lambda_{1,D}^d$ has only one nodal domain, so the maximum principle (e.g.\ applying \cite[Chapter 2, Theorem 5]{protter-weinberger} in local coordinates) implies that $u$ is strictly positive or strictly negative on $M \backslash \partial M$. Hence there cannot be two orthogonal Dirichlet eigenfunctions of $\Delta^d$ corresponding to $\lambda_{1,D}^d$, i.e.\ $\lambda_{1,D}^d$ is a simple eigenvalue of $\Delta^d$, and \eqref{eq:dfirstcheegerd} follows from Theorem \ref{thm:dlevelset}. 
In the Neumann case, Corollary \ref{thm:dcourant} yields that any dynamic Neumann eigenfunction $u$ of $\Delta^d$ corresponding to $\lambda_{2,N}^d$ has at most two nodal domains. Since the constant function $\mathbf{1}$ is a dynamic Neumann eigenfunction of $\Delta^d$ orthogonal to $u$, $u$ has exactly two nodal domains $G_1,G_2$. One choice of $G \in \{G_1,G_2\}$ satisfies $\mu(G) \le \mu(M \backslash \overline{G})$, and \eqref{eq:dfirstcheegern} follows from Theorem \ref{thm:dlevelset}. 
\end{proof}

\subsubsection{Higher dynamic Cheeger inequalities} 

We can extend our {higher} Cheeger inequalities of Section \ref{sec:hcheeg} directly to the dynamic setting. Our proofs of Theorem \ref{thm:cheeger} and Proposition \ref{thm:seba} carry over directly to the dynamic setting (Theorem \ref{thm:dcheeger} and Proposition \ref{thm:dseba}). To extend Theorems \ref{thm:miclocheeger} and \ref{thm:miclorcheeger} to the dynamic setting, we can avoid some technicalities by applying those theorems on the geometry of mixing manifold $(M_0,\bar{g},\mu_0)$, and applying \eqref{eq:schillingn}. 

\begin{theorem} \label{thm:dmiclo}
There is a universal constant $\hat{\eta}$ such that for any dynamical system where $M_0$ is boundaryless, for all $k \ge 1$ we have 
\begin{align}
    \lambda_{k,\emptyset}^d \le -\frac{\hat{\eta}}{k^6} (h_{k,\emptyset}^d)^2. \label{eq:dmiclo}
\end{align}
\end{theorem}

\begin{proof}
By Proposition \ref{thm:geommixing}, $\lambda_{k,\emptyset}^d$ is the $k$th eigenvalue of $\Delta_{\bar{g},\mu_0}$. Applying Theorem \ref{thm:miclocheeger} to bound the $k$th Cheeger constant $h_{k,\emptyset}$ on the geometry of mixing manifold yields $\lambda_{k,\emptyset}^d \le -\frac{\hat{\eta}}{k^6} h_{k,\emptyset}^2$. Then \eqref{eq:schillingn} and the definitions \eqref{eq:defcheegnk} and \eqref{eq:dcheegn} imply $-h_{k,\emptyset} \le -h_{k,\emptyset}^d$, and \eqref{eq:dmiclo} follows. 
\end{proof}

\begin{theorem} \label{thm:dmiclor}
There is a universal constant $\eta$ such that for any dynamical system $\mathcal{T}$ where $M_0$ is boundaryless, for all $k \ge 1$ we have
\begin{align}
    \lambda_{2k,\emptyset}^d \le -\frac{\eta }{\log(k+1)} (h_{k,\emptyset}^d)^2.
\end{align}
\end{theorem}

\begin{proof}
By Proposition \ref{thm:geommixing}, $\lambda_{2k,\emptyset}^d$ is the $2k$th eigenvalue of $\Delta_{\bar{g},\mu_0}$. Applying Theorem \ref{thm:miclorcheeger} to bound the $k$th Cheeger constant $h_{k,\emptyset}$ on the geometry of mixing manifold yields $\lambda_{2k,\emptyset}^d \le -\frac{\eta}{\log(k+1)} h_{k,\emptyset}^2$. Then \eqref{eq:schillingn} and the definitions \eqref{eq:defcheegnk} and \eqref{eq:dcheegn} imply $-h_{k,\emptyset} \le -h_{k,\emptyset}^d$, and \eqref{eq:dmiclo} follows. 
\end{proof}

Our constructive, nodal domain-based higher Cheeger inequality, Theorem \ref{thm:cheeger}, generalises directly to the dynamic case. 

%
\begin{theorem}[Higher dynamic Cheeger inequality] \label{thm:dcheeger}
Let $\mathcal{T}$ be a dynamical system. For each $k \ge 1$, let $r_k$ be the maximal number of nodal domains in any dynamic Neumann (resp.\ Dirichlet) eigenfunction $u$ of $\Delta^d$ with eigenvalue $\lambda\ge \lambda_{k,N}^d$ (resp.\ $\lambda \ge \lambda_{k,D}^d$). 
\begin{enumerate}
\item We have 
\begin{align}
    \lambda_{k,N}^d &\le -\frac{1}{4} (h_{r_k,N}^d)^2, \label{eq:dcheegerr} \\
    \lambda_{k,D}^d &\le -\frac{1}{4} (h_{r_k,D}^d)^2. \label{eq:dcheegerrd}
\end{align}

\item Let $u$ be an eigenfunction with eigenvalue $\lambda \ge \lambda_{k,N}^d$ (resp.\ $\lambda \ge \lambda_{k,D}^d$) and with $r_k$ nodal domains. Let $G^1,\ldots,G^{r_k} \subset M$ denote the nodal domains of $u$, and for each $i$ and each $s \in \range(u^2|_{G^i})$, let $G^i_s$ denote the $s$-superlevel set of $u^2$ on $G^i$. For each $i$, define $S_{G^i}$ as in \eqref{eq:dSdef} or \eqref{eq:dSdefd}. Then each $S_{G^i}$ has positive Lebesgue measure satisfying \eqref{eq:dSmeasure}, and for each $\{s_1,\ldots,s_{r_k}\} \in S_{G^1} \times \ldots \times S_{G^{r_k}}$, the collection $\mathcal{A}_{r_k}:=\{G_{s_1}^1,\ldots,G_{s_{r_k}}^{r_k}\}$ is a Neumann (resp.\ Dirichlet) $r_k$-packing of $M_0$ satisfying $\lambda_{k,N}^d \le -\frac{1}{4} \mathcal{J}_N^d(\mathcal{A}_{r_k})^2$ (resp.\ $\lambda_{k,D}^d \le -\frac{1}{4} \mathcal{J}_D^d(\mathcal{A}_{r_k})^2$). 
\end{enumerate}
\end{theorem}

\begin{proof}
    This theorem follows from Lemma \ref{thm:dlevelset}, by the reasoning in the proof of Theorem \ref{thm:cheeger}. 
\end{proof}

We can also extend Proposition \ref{thm:seba} to the dynamic setting, to obtain bounds on $h_{l,N}^d$ or $h_{l,D}^d$ for $r_k \le l \le k$ in terms of thresholded functions obtained from linear combinations of the first $k$ eigenfunctions of $\Delta^d$. 

%
\begin{proposition} \label{thm:dseba}
For any dynamical system $\mathcal{T}$, let $u_1,\ldots,u_k$ denote the first $k$ dynamic Neumann, resp.\ Dirichlet, eigenfunctions of $\Delta^d$ for $k \ge 1$. For any $1 \le l \le k$ and any $\alpha \in \mathbb{R}^{l\times k}$, define $f_{1,\alpha},\ldots,f_{l,\alpha}$ by $f_{i,\alpha}:=\sum_{j=1}^k \alpha_{ij}u_j$. Suppose that for some $a>0$, the functions $\tau_a(f_{1,\alpha}),\ldots,\tau_a(f_{l,\alpha})$ are nonzero and have pairwise disjoint supports. 
Then each $\tau_a(f_{i,\alpha})$ has a nodal domain $\tilde{G}^i$ such that letting $\tilde{G}^i_s$ for $s \in \range(\tau_a(f_{i,\alpha})^2|_{\tilde{G}^i})$ denote the $s$-superlevel set of $\tau_a(f_{i,\alpha})^2$ on $\tilde{G}^i$, the set 
\begin{align}
\tilde{S}_{\tilde{G}^i} := 
\Bigl\{ s \in \,&\range(\tau_a(f_{i,\alpha})^2|_{\tilde{G}^i}) : \tilde{G}^i_s \in \mathscr{P}_N(M_0), \nonumber \\
&\quad \frac{\sum_{t=0}^{t_{\max}} \||\nabla_{g_t} \Phi^{(t)}_* \tau_a(f_{i,\alpha})|\|_{L^2(\Phi^{(t)}(\tilde{G}^i);\mu_t)}^2}{|\mathrm{T}| \|\tau_a(f_{i,\alpha})\|_{L^2(\tilde{G}^i;\mu_0)}^2} \ge \frac{1}{4} \mathcal{J}_N^d(\tilde{G}^i_s)^2 \Bigr\}, \label{eq:dsebasn} 
\end{align}
resp.\ 
\begin{align}
\tilde{S}_{\tilde{G}^i} := 
\Bigl\{ s \in \,&\range(\tau_a(f_{i,\alpha})^2|_{\tilde{G}^i}) : \tilde{G}^i_s \in \mathscr{P}_D(M_0), \nonumber \\
&\quad \frac{\sum_{t=0}^{t_{\max}} \||\nabla_{g_t} \Phi^{(t)}_* \tau_a(f_{i,\alpha})|\|_{L^2(\Phi^{(t)}(\tilde{G}^i);\mu_t)}^2}{|\mathrm{T}| \|\tau_a(f_{i,\alpha})\|_{L^2(\tilde{G}^i;\mu_0)}^2} \ge \frac{1}{4} \mathcal{J}_D^d(\tilde{G}_s^i)^2 \Bigr\}, \label{eq:dsebasd}
\end{align}
has positive measure and satisfies {
\begin{align}
    \Leb(\tilde{S}_{\tilde{G}^i}) \ge \frac{\|\overline{\tilde{h}_i^d} - \tilde{\mathrm{h}}_i^d\|_{L^1\mleft(\range\mleft(\tau_a(f_{i,\alpha})^2|_{\tilde{G}^i}\mright);\tilde{\mathbb{P}}_i\mright)} \|\tau_a(f_{i,\alpha})\|_{L^2(\tilde{G}^i;\mu_0)}^2}{2\mleft(\overline{\tilde{h}_i^d} - \inf_{s \in \smash{\range\mleft(\tau_a(f_{i,\alpha})^2|_{\tilde{G}^i}\mright)}} \tilde{\mathrm{h}}^d_i(s) \mright) \mu_0(\tilde{G}^i)}, \label{eq:dseba-Smeasure}
\end{align}
where $\overline{\tilde{h}_i^d}$, $\tilde{\mathrm{h}}_i^d$ and $\tilde{\mathbb{P}}_i^d$ are defined in the proof below}. Moreover, for each $\{s_1,\ldots,s_l\} \in \tilde{S}_{\tilde{G}^1} \times \ldots \times \tilde{S}_{\tilde{G}^l}$, the collection $\mathcal{A}_l:=\{\tilde{G}_{s_1}^1,\ldots,\tilde{G}_{s_l}^l\}$ is a Neumann $l$-packing for $M_0$ satisfying 
\begin{align}
    \lambda_{k,N}^d 
    \le -\frac{1}{4} \mathcal{J}_N^d(\mathcal{A}_l)^2 \max_{1 \le j \le l} \frac{\|\tau_a(f_{j,\alpha}) \|_{L^2(M_0;\mu_0)}^2}{\|f_{j,\alpha}\|_{L^2(M_0;\mu_0)}^2} 
    \le -\frac{1}{4} (h_{l,N}^d)^2                                \max_{1 \le j \le l} \frac{\|\tau_a(f_{j,\alpha}) \|_{L^2(M_0;\mu_0)}^2}{\|f_{j,\alpha}\|_{L^2(M_0;\mu_0)}^2}, \label{eq:dsebacheeger}
\end{align}
resp.\ a Dirichlet $l$-packing for $M_0$ satisfying 
\begin{align}
    \lambda_{k,D}^d 
    \le -\frac{1}{4} \mathcal{J}_D^d(\mathcal{A}_l)^2 \max_{1 \le j \le l} \frac{\|\tau_a(f_{j,\alpha}) \|_{L^2(M_0;\mu_0)}^2}{\|f_{j,\alpha}\|_{L^2(M_0;\mu_0)}^2} 
    \le -\frac{1}{4} (h_{l,D}^d)^2                                \max_{1 \le j \le l} \frac{\|\tau_a(f_{j,\alpha}) \|_{L^2(M_0;\mu_0)}^2}{\|f_{j,\alpha}\|_{L^2(M_0;\mu_0)}^2}.
\end{align}
\end{proposition}
\begin{proof}
This result follows by the reasoning for Proposition \ref{thm:seba} and Lemma \ref{thm:dlevelset}. As in those proofs, we consider only the Neumann case. 
For each $1 \le i \le l$, we select $\tilde{G}^i$ by $\tilde{G}^i:= 
\argmin_{\tilde{G}} \frac{\sum_{t=0}^{t_{\max}} \||\nabla_{g_t} \Phi^{(t)}_* \tau_a(f_{i,\alpha})|\|_{L^2(\Phi^{(t)}(\tilde{G});\mu_t)}^2}{|\mathrm{T}| \|\tau_a(f_{i,\alpha})\|_{L^2(\tilde{G};\mu_0)}^2}$, where the infimum is taken over nodal domains $\tilde{G}$ of $\tau_a(f_{i,\alpha})$. Then the reasoning for Theorem \ref{thm:levelset}, modified as in the proofs of Proposition \ref{thm:seba} and Theorem \ref{thm:dlevelset}, imply that $\tilde{S}_{\tilde{G}^i}$ has positive measure. The reasoning for \eqref{eq:rayleigh-lambda-seba} extends directly to the dynamic setting, and \eqref{eq:dsebacheeger} follows as in the proof of Proposition \ref{thm:seba}. 

Now, define $\overline{\tilde{h}_i^d}:= \frac{1}{\|\tau_a(f_{i,\alpha})\|_{L^2(M_0;\mu_0)}^2} \int_{\range\mleft(\tau_a(f_{i,\alpha})^2|_{\tilde{G}^i}\mright)} \mathcal{J}_N^d(\tilde{G}_s^i) \mu_0(\tilde{G}_s^i) \dd s$ and $\tilde{\mathrm{h}}_i^d(s):=\mathcal{J}_N^d(\tilde{G}_s^i)$, and define the probability measure $\tilde{\mathbb{P}}_i$ on $\range(\tau_a(f_{i,\alpha})^2|_{\tilde{G}^i})$ by $\tilde{\mathbb{P}}_i(L):=\int_L \frac{\mu_0(\tilde{G}_s^i)}{\|\tau_a(f_{i,\alpha})\|_{L^2(\tilde{G}^i;\mu_0)}^2} \dd s$. Then the reasoning for \eqref{eq:dSmeasure} implies {\eqref{eq:dseba-Smeasure}}.
A similar bound holds in the Dirichlet case, replacing $\mathcal{J}_N^d$ with $\mathcal{J}_D^d$ in the definitions of $\overline{\tilde{h}_i^d}, \tilde{\mathrm{h}}_i^d, \tilde{\mathbb{P}}_i$. 
\end{proof}


%% file: 05_examples.tex

\section{Examples} \label{sec:cheeger-examples}
We apply our higher Cheeger inequality (Theorem \ref{thm:cheeger}) to compare the Laplace-Beltrami eigenvalues to the higher Cheeger constants, on three manifolds: a torus (example \ref{sec:torus}), a cylinder using Neumann boundary conditions (example \ref{sec:cylinder}) and a 3-ball using Dirichlet boundary conditions (example \ref{sec:3ball}). 
Our Theorem \ref{thm:cheeger} applies to manifolds with or without boundary, whenever we know the number of nodal domains in some eigenfunctions on those manifolds, i.e.\ to each of examples \ref{sec:torus}--\ref{sec:3ball}. Miclo's existing higher Cheeger inequalities (Theorems \ref{thm:miclocheeger} and \ref{thm:miclorcheeger}) apply only to manifold without boundary, i.e.\ to example \ref{sec:torus}. For that example, we obtain an asymptotically stronger bound on $h_{k,\emptyset}$ using our Theorem \ref{thm:cheeger} than using Miclo's Theorems \ref{thm:miclocheeger} and \ref{thm:miclorcheeger}. 
Using our higher dynamic Cheeger inequality (Theorem \ref{thm:dcheeger}), we also compare the dynamic Laplacian eigenvalues to the dynamic Cheeger constants for one dynamical system, a cylinder with linear shear (example \ref{sec:shear}). 

\subsection{Cheeger constants on a torus} \label{sec:torus}

Our first example is a flat torus $\mathbb{T}^2:=2\pi \mathbb{S}^1 \times 2\pi \mathbb{S}^1$, endowed with two-dimensional Lebesgue measure. Then $\Delta$ has an orthogonal Hilbert basis of eigenfunctions on $L^2(\mathbb{T}^2,\mathrm{Leb}^{{2}})$, consisting of all functions of the form 
\begin{align}
    u_{k_1,k_2,\zeta_1,\zeta_2}(x,y)&:=\cos(k_1(x+\zeta_1))\cos(k_2(y+\zeta_2)), \label{eq:torus-efunc}
\end{align}
for $k_1,k_2=0,1,2,\ldots$ and $\zeta_1,\zeta_2 \in \{0,\frac{\pi}{2}\}$, where we require $\zeta_1=0$ if $k_1=0$ and $\zeta_2=0$ if $k_2=0$ to ensure an orthogonal basis. Each eigenfunction $u_{k_1,k_2,\zeta_1,\zeta_2}$ has corresponding eigenvalue $\lambda_{k_1,k_2,\zeta_1,\zeta_2}=-k_1^2 - k_2^2$, and we can globally order these eigenfunctions in order of decreasing eigenvalue (resolving ties arbitrarily). 

To apply Theorem \ref{thm:cheeger}, we need to estimate the maximal number $r_k$ of nodal domains of an eigenfunction with eigenvalue greater than or equal to the $k$th eigenvalue $\lambda_{k,\emptyset}$. Each eigenfunction $u_{k_1,k_2,\zeta_1,\zeta_2}$ has $\max\{4k_1k_2,2k_1,2k_2,1\}$ nodal domains, by \eqref{eq:torus-efunc}. It can be shown that for each $k_1\ge 1$ and $\zeta_1,\zeta_2 \in \{0,\frac{\pi}{2}\}$, any eigenfunction whose eigenvalue is greater than or equal to $\lambda_{k_1,k_1,\zeta_1,\zeta_2}$ has at most $4k_1^2$ nodal domains. In this sense, the eigenfunctions $u_{k_1,k_1,\zeta_1,\zeta_2}$ maximise the number of nodal domains of an eigenfunction under an eigenvalue constraint. Thus, noting that $\lambda_{6,\emptyset}=\lambda_{1,1,\zeta_1,\zeta_2}$, we can obtain a lower bound on $r_k$ for any $k \ge 6$ by finding the largest $k_1$ such that $\lambda_{k,\emptyset} \le \lambda_{k_1,k_1,\zeta_1,\zeta_2}$ for some $\zeta_1,\zeta_2 \in \{0,\frac{\pi}{2}\}$, and noting that $r_k$ is bounded below by the number of {nodal domains} in $u_{k_1,k_2,\zeta_1,\zeta_2}$. 
To estimate this $k_1$ in terms of $k$, we note that $\lambda_{k,\emptyset} \ge \lambda_{k_1+1,k_1+1,\zeta_1,\zeta_2}=-2(k_1+1)^2$. 
Now, each integer pair in $\mathcal{I}:=\{(i_1,i_2)\in \mathbb{Z}^2 : i_1,i_2 \ge 1,-i_1^2-i_2^2 \ge -2(k_1+1)^2\}$ corresponds to a unit-area square contained entirely in the nonnegative quadrant $Q$ of the disk $\{-x^2-y^2\ge -2(k_1+1)^2\}$. The quadrant $Q$ has area $\frac{\pi}{2}(k_1+1)^2$, so we have $|\mathcal{I}| \le \frac{\pi}{2} (k_1+1)^2$.  
Each integer pair in $\mathcal{I}$ corresponds to 4 linearly independent eigenfunctions of the form \eqref{eq:torus-efunc} with different choices of $\zeta_1,\zeta_2 \in \{0,\frac{\pi}{2}\}$, leading to at most $2\pi(k_1+1)^2$ eigenvalues, counted with multiplicity, greater than or equal to $\lambda_{k_1+1,k_1+1,\zeta_1,\zeta_2}$. 

There are also $2\lfloor \sqrt{2}(k_1+1)\rfloor + 1$ integer pairs in $\mathcal{I}':=\{(i_1,i_2) \in \mathbb{Z}^2 : i_1,i_2 \ge 0,i_1i_2=0,-i_1^2-i_2^2 \ge -2(k_1+1)^2\}$. 
Each such integer pair with $i_1\ge 1$ or $i_2 \ge 1$ corresponds to 2 linearly independent eigenfunctions of the form \eqref{eq:torus-efunc} with different choices of $\zeta_1 \in \{0,\frac{\pi}{2}\}$ or $\zeta_2 \in \{0,\frac{\pi}{2}\}$ respectively, while the pair $(0,0)$ corresponds to only 1 eigenfunction. This leads to an additionally $4\lfloor \sqrt{2}(k_1+1)\rfloor + 1$ additional eigenvalues greater than or equal to $\lambda_{k_1+1,k_1+1,\zeta_1,\zeta_2}$. In total, we have at most $2\pi(k_1+1)^2+4\sqrt{2}(k_1+1)+1$ eigenvalues greater than or equal to $\lambda_{k_1+1,k_1+1,\zeta_1,\zeta_2}$. The ordering of the eigenvalues $\lambda_{i,\emptyset}$ implies there are at least $k$ eigenvalues greater than or equal to $\lambda_{k_1+1,k_1+1,\zeta_1,\zeta_2}$, so $k \le 2\pi(k_1+1)^2+4\sqrt{2}(k_1+1)+1$. Applying the quadratic formula and noting $\sqrt{\pi k + 4 - \pi} \ge \sqrt{\pi k}$ yields the bound $k_1 \ge \sqrt{\frac{k}{2\pi}}-1-\frac{\sqrt{2}}{\pi}$. Now, $u_{k_1,k_1,\zeta_1,\zeta_2}$ has $4k_1^2$ nodal domains, so this bound on $k_1$ and the fact $k_1 \ge 1$ imply $r_k \ge 4k_1^2 \ge \max\mleft\{ \frac{2k}{\pi}-4.7\sqrt{k}+8.4,4 \mright\}\approx 0.64k+O(k)$. {For comparison, Pleijel's nodal domain theorem (e.g.\ \cite[Theorem 1.2]{lena}) yields $\limsup_{k \to \infty} \frac{r_k}{k} \le \frac{4}{j_{0,1}^2}\approx 0.69$ where $j_{0,1}$ denotes the first zero of the zeroth Bessel function of the first kind. That is, these eigenfunctions $u_{k_1,k_1,\zeta_1,\zeta_2}$ have close to the maximum possible number of nodal domains for their position in the eigenfunction order, asymptotically for large $k$.
Thus, Theorem \ref{thm:cheeger} is well suited to this manifold.} Theorem \ref{thm:cheeger} implies 
\begin{align}
    \lambda_{k,\emptyset} \le -\frac{1}{4} h_{r_k,\emptyset}^2 \le -\frac{1}{4} h_{\max \mleft\{ \mleft \lceil \frac{2k}{\pi}-4.7\sqrt{k}+8.4\mright \rceil, 4 \mright\},\emptyset}^2. \label{eq:torus-ndcheegerm}
\end{align}

To compare \eqref{eq:torus-ndcheegerm} to the bounds from Miclo's Theorems \ref{thm:miclocheeger} and \eqref{eq:miclorcheeger}, we rewrite the outer inequality of \eqref{eq:torus-ndcheegerm} as a bound on $h_{l,\emptyset}$ for $l \ge 1$, and use Weyl's law. Let $k^*(l):=\lceil \frac{\pi l}{2} + 9.3\sqrt{l+0.3}+14.2 \rceil$, then we can rearrange \eqref{eq:torus-ndcheegerm} to obtain
\begin{align}
    h_{l,\emptyset} \le 2 \sqrt{-\lambda_{k^*(l),\emptyset}}. \label{eq:torus-ndcheegerl}
\end{align}
Now, from Weyl's law (see e.g.\ \cite[p.118]{lablee}), it follows that 
\begin{align}
    \lambda_{k,\emptyset}=-\frac{k}{\pi}+O(\sqrt{k}). \label{eq:weyl}
\end{align}
This allows us to compare our bound \eqref{eq:torus-ndcheegerl} with the bounds obtained from Miclo's Theorems \ref{thm:miclocheeger} and \ref{thm:miclorcheeger}. 
\begin{itemize}
\item 
Substituting \eqref{eq:weyl} and the definition of $k^*(l)$ into our bound \eqref{eq:torus-ndcheegerl}, we obtain that as $l \to \infty$, 
\begin{align}
    h_{l,\emptyset} \le 2\sqrt{\frac{l}{2}+O(\sqrt{l})}=2\sqrt{\frac{l}{2}}+O(1). \label{eq:torus-ndweyl}
\end{align}
\item 
Substituting \eqref{eq:weyl} into Miclo's Theorem \ref{thm:miclocheeger} \cite[Theorem 7]{miclo15}, the reasoning from \eqref{eq:torus-ndweyl} implies that as $l \to \infty$, 
\begin{align}
    h_{l,\emptyset} \le l^3 \sqrt{-\frac{\lambda_{l,\emptyset}}{\hat{\eta}}}=l^3\sqrt{\frac{l}{\pi \hat{\eta}}}+O(l^3). \label{eq:torus-micloweyl}
\end{align}
This is clearly asymptotically weaker than \eqref{eq:torus-ndweyl}.
\item 
Substituting \eqref{eq:weyl} into Miclo's Theorem \ref{thm:miclorcheeger} \cite[Theorem 13]{miclo15}, the reasoning from \ref{eq:torus-ndweyl} implies that as $l \to \infty$, 
\begin{align}
    h_{l,\emptyset} \le \sqrt{-\frac{\log(2l+1)\lambda_{2l,\emptyset}}{\eta}}=\sqrt{\frac{2l\log(2l+1)}{\pi \eta}}+O(\sqrt{\log(2l+1)}). \label{eq:torus-miclorweyl}
\end{align}
This is also asymptotically weaker than \eqref{eq:torus-ndweyl}. 
\end{itemize}

{As a partial converse to Theorem \ref{thm:cheeger}, the higher Buser inequality \cite[Theorem 4.1]{liu} states 
\begin{align}
    h_{l,\emptyset} \ge \frac{1}{57l} \sqrt{\lambda_{l,\emptyset}} \label{eq:hbuser}
\end{align}
on $\mathbb{T}^2$. Substituting \eqref{eq:weyl} into \eqref{eq:hbuser} yields 
\begin{align}
    h_{l,\emptyset} \ge \frac{\sqrt{\frac{l}{\pi}+O(\sqrt{l})}}{57l}=\frac{1}{57\sqrt{\pi l}}+O\mleft(\frac{1}{l}\mright).
\end{align}
}
\subsection{Cheeger constants of a cylinder} \label{sec:cylinder}


Next, we consider a cylinder $\mathcal{C}:=2\pi \mathbb{S}^1 \times [0,\pi]$, endowed with two-dimensional Lebesgue measure. Then $\mathcal{C}$ is a semiconvex subset of the torus $\mathbb{T}^2$ from example \ref{sec:torus}, but $\mathcal{C}$ is not a convex subset of any manifold since some pairs of points in $\mathcal{C}$ are connected by two minimal geodesics contained in $\mathcal{C}$. Under Neumann boundary conditions, $\Delta$ has an orthogonal Hilbert basis of eigenfunctions on $L^2(\mathcal{C},\mathrm{Leb}^{{2}})$, consisting of all functions of the form 
\begin{align}
    u_{k_1,k_2,\zeta}(x,y)&:= \cos(k_1 (x + \zeta))\cos(k_2 y), \label{eq:cylinder-efunc}
\end{align}
for $k_1,k_2=0,1,2,\ldots$ and $\zeta \in \{0,\frac{\pi}{2}\}$, where we require $\zeta=0$ whenever $k_1=0$ to ensure an orthogonal basis. 
Each eigenfunction $u_{k_1,k_2,\zeta}$ has corresponding eigenvalue $\lambda_{k_1,k_2,\zeta}=-k_1^2 - k_2^2$. 
To apply Theorem \ref{thm:cheeger}, we again need a lower bound for $r_k$. 
First, we show that for each $k_1 \ge 1$, eigenfunctions of the form $u_{k_1,k_1,\zeta}$  have the maximal number of nodal domains, among eigenfunctions of the form \eqref{eq:cylinder-efunc} for which $\lambda_{i_1,i_2,\zeta}\ge -2k_1^2$. Each $u_{i_1,i_2,\zeta}$ has $(i_2+1)\max\{2i_1,1\}$ nodal domains by \eqref{eq:cylinder-efunc}, so maximising the number of nodal domains in $u_{i_1,i_2,\zeta}$ subject to $\lambda_{i_1,i_2,\zeta} \, (=-i_1^2 - i_2^2) \ge -2k_1^2$ is equivalent to solving $\max\{2i_1(i_2+1) : (i_1,i_2) \in \mathbb{Z}^2_{\ge 0}, i_1^2+i_2^2 \le 2k_1^2\}$. This can be solved via the relaxation $\max \{2x(y+1) : (x,y) \in ([0,k_1] \cup [{k_1+1},\infty)) \times \mathbb{R}_{\ge 0}, x^2+y^2 \le 2k_1^2\}$. Rearranging the constraint $x^2+y^2 \le 2k_1^2$ and maximising $y$ gives us $y=\sqrt{2k_1^2-x^2}$. Substituting this into $2x(y+1)$ gives us $2x(\sqrt{2k_1^2-x^2}+1)$, which is strictly increasing for $0 \le x \le k_1$ and strictly decreasing for $k_1+1 \le x \le \sqrt{2}k_1$. Thus, since the objective is larger at $(x,y)=(k_1,k_1)$ than at $(x,y)=(k_1+1,\sqrt{k_1^2-2k_1-1})$, the maximum is uniquely attained at $(x,y)=(k_1,k_1)$. Hence the eigenfunctions $u_{k_1,k_1,\zeta}$ for $\zeta \in \{0,\frac{\pi}{2}\}$ maximise the number of nodal domains, among eigenfunctions $u_{i_1,i_2,\zeta}$ of the form \eqref{eq:cylinder-efunc} satisfying $\lambda_{i_1,i_2,\zeta} \ge -2k_1^2$. 

Now, we bound $r_k$ for each $k \ge 5$ by finding the largest $k_1$ such that $\lambda_{k,N} \le \lambda_{k_1,k_1,\zeta}$ for $\zeta \in \{0,\frac{\pi}{2}\}$, noting that $\lambda_{5,N}=\lambda_{1,1,\zeta}$. 
For this $k_1$, we have $\lambda_{k,N} \ge \lambda_{k_1+1,k_1+1,\zeta}=-2(k_1+1)^2$. 
Each integer pair in the set $\mathcal{I}$ from the previous example corresponds to two linearly independent eigenfunctions of the form \eqref{eq:cylinder-efunc}, leading to at most $\lfloor \pi(k_1+1)^2\rfloor$ eigenvalues $\ge \lambda_{k_1+1,k_1+1,\zeta}$. There are also $\lfloor \sqrt{2}(k_1+1) \rfloor$ nonnegative integer pairs in $\mathcal{I}'$ from the previous example with $i_1>0$, each corresponding to 2 linearly independent eigenfunctions, and $\lfloor \sqrt{2}(k_1+1)\rfloor+1$ such pairs with $i_1=0$, each corresponding to only 1 linearly independent eigenfunction. These lead to at most an additional $3\sqrt{2}(k_1+1)+1$ eigenvalues $\ge \lambda_{k_1+1,k_1+1,\zeta}$. Thus, there are at most $\pi(k_1+1)^2+3\sqrt{2}(k_1+1)+1$ eigenvalues $\ge \lambda_{k_1+1,k_1+1,\zeta}$. 
Again, the ordering of the $\lambda_{i,\emptyset}$ implies there are at least $k$ eigenvalues $ge \lambda_{k_1+1,k_1+1,\zeta}$, so $k \le \pi(k_1+1)^2 + 3\lfloor \sqrt{2} (k_1+1)\rfloor + 1$. Then the quadratic formula and the fact $\sqrt{4\pi k + 18 - 4\pi} \ge \sqrt{4\pi k}$ yield $k_1 \ge \sqrt{\frac{k}{\pi}}-1-\frac{3}{\sqrt{2}\pi}$. Now, $u_{k_1,k_1,\zeta}$ has $2k_1(k_1+1)$ nodal domains, so this bound on $k_1$ and the fact $k_1 \ge 1$ imply $r_k \ge 2k_1(k_1+1) \ge \max \mleft\{ \frac{2k}{\pi}-2.7\sqrt{k}+2.2, 4 \mright\}$. Thus, Theorem \ref{thm:cheeger} implies that for $k \ge 5$, 
\begin{align}
    \lambda_{k,N} \le -\frac{1}{4} h_{r_k,N}^2 \le -\frac{1}{4} h_{\max \mleft\{ \mleft \lceil \frac{2k}{\pi}-2.7\sqrt{k}+2.2\mright \rceil, 4 \mright\}, N}^2. \label{eq:cylinder-ndcheegerm}
\end{align}
Note that we cannot apply Miclo's Theorems \ref{thm:miclocheeger} or \ref{thm:miclorcheeger} to $\mathcal{C}$, because $\mathcal{C}$ has nonempty boundary. 

\subsection{Cheeger constants on a 3-ball} \label{sec:3ball}

Next, we consider the 3-ball $\mathbb{B}:=\{\mathbf{x} \in \mathbb{R}^3 : |\mathbf{x}|\le 1\}$, equipped with 3-dimensional Lebesgue measure. We work in spherical coordinates $(r,\theta,\phi)$, where $\theta$ is the polar angle and $\phi$ is the azimuthal angle. Then $\Delta$, under Dirichlet boundary conditions, has an orthogonal Hilbert basis of eigenfunctions on $L^2(\mathbb{B},\mathrm{Leb}^{{3}})$, consisting of all functions of the form 
\begin{align}
    u_{k_1,k_2,k_3,\zeta}&:=S_{k_2} \mleft(\alpha_{k_1,k_2} r \mright) P^{k_3}_{k_2} (\cos \theta) \cos(k_3(\phi+\zeta)) \label{eq:3ball-efunc}
\end{align}
for $k_1=1,2,\ldots$; $k_2=0,1,\ldots$; $k_3=0,\ldots,k_2$; $\zeta\in\{0,\frac{\pi}{2}\}$, where we require $\zeta=0$ when $k_3=0$ to ensure an orthonormal basis. The function $S_{k_2}:\mathbb{R}_+ \to \mathbb{R}$ is the $k_2$th \emph{spherical Bessel function of the first kind}, $\alpha_{k_1,k_2}$ is the $k_1$th positive zero of $S_{k_2}$, and $P^{k_3}_{k_2}$ is the $k_2$th \emph{associated Legendre polynomial of $k_3$th order} (see e.g.\ \cite[sec.\ 3.3]{grebenkov-nguyen} and \cite[secs\ V.8 and VII.5]{courant}). The eigenfunction $u_{k_1,k_2,k_3,\zeta}$ has eigenvalue $\lambda_{k_1,k_2,k_3,\zeta}=-\alpha_{k_1,k_2}^2$. The values $\alpha_{k_1,k_2}$ satisfy the bounds (simplified from \cite[equations (1),\ (2),\ (5)]{breen})
\begin{align}
    \pi k_1 + k_2 - 3.75 < \alpha_{k_1,k_2} < \pi k_1 + \frac{\pi}{2}k_2 + 0.03-\frac{(k_2+\frac{1}{2})^2}{2\mleft(\pi k_1 + \frac{\pi}{2}k_2 + 0.03\mright)}. \label{eq:besselzerobound}
\end{align}

To apply our Theorem \ref{thm:cheeger}, we first obtain a lower bound on $r_k$. The function $P_{k_2}^{k_3}(\cos \theta)\cos(k_3(\phi+\zeta))$ has $(k_2-k_3+1)\max\{2k_3,1\}$ nodal domains (see e.g.\ \cite[p.302]{leydold}), while the function $S_{k_2}(\alpha_{k_1,k_2} r)$ has $k_1$ nodal domains since $\alpha_{k_1,k_2}$ is the $k_1$th positive zero of $S_{k_2}$. Thus, the eigenfunction $u_{k_1,k_2,k_3,\zeta}$ has $k_1(k_2-k_3+1)\max\{2k_3,1\}$ nodal domains. In particular, $u_{k_1,4k_1-1,2k_1,\zeta}$ for $k_1=1,2,\ldots$, $\zeta \in \{0,\frac{\pi}{2}\}$, has $8k_1^3$ nodal domains, i.e.\ it is a simple eigenfunction with a relatively high number of nodal domains for its eigenvalue. It can be shown using the second inequality in \eqref{eq:besselzerobound} that with $c:=3\pi-\frac{8}{3\pi}$, 
\begin{align}
    \lambda_{k_1,4k_1-1,2k_1,\zeta} = -\alpha_{k_1,4k_1-1}^2 \ge -(c k_1-1.46)^2. \label{eq:3ball-lambda}
\end{align}
Thus, for each $k \ge 18$, we can obtain a lower bound on $r_k$ by finding the largest $k_1$ such that 
\begin{align}
    -(c k_1-1.46)^2 \ge \lambda_{k,D}, \label{eq:3ball-k1def}
\end{align}
since we can confirm numerically that $\lambda_{17,D} \ge -(c-1.46)^2 \ge \lambda_{18,D}$. For this $k_1$, we have $\lambda_{k,D} \ge -(c (k_1+1) - 1.46)^2$. By the first inequality in equation \eqref{eq:besselzerobound}, we have $\lambda_{i_1,i_2,i_3,\zeta} \ge -(c (k_1+1)-1.46)^2$ only for $(i_1,i_2,i_3,\zeta) \in \mathcal{I}:=\{(i_1,i_2,i_3,\zeta) : \pi i_1 + i_2 - 3.75 \le c (k_1+1)-1.46\}$. There are $2i_2+1$ tuples $(i_1,i_2,i_3,\zeta) \in \mathcal{I}$ for each pair $i_1,i_2$ such that $\pi i_1 + i_2 \le c(k_1+1)+2.29$. Using the formula for sums of squares, and writing $a:=c(k_1+1)+2.29$ for clarity, the cardinality of $\mathcal{I}$ is bounded by 
\begin{align}
    |\mathcal{I}| &= \sum_{i_1=1}^{\mleft\lfloor \frac{a}{\pi} \mright\rfloor} \sum_{i_2=0}^{\lfloor a - \pi i_1 \rfloor} (2i_2+1) 
    = \sum_{i_1=1}^{\mleft\lfloor \frac{a}{\pi} \mright\rfloor} (\lfloor a - \pi i_1 \rfloor + 1)^2 
    = \sum_{i_1=1}^{\mleft\lfloor \frac{a}{\pi} \mright\rfloor} \mleft( \mleft\lfloor a - \pi \mleft(\mleft\lfloor \frac{a}{\pi} \mright\rfloor + 1 - i_1 \mright) \mright\rfloor + 1 \mright)^2 \nonumber \\
    &\le \sum_{i_1=1}^{\mleft\lfloor \frac{a}{\pi} \mright\rfloor} (\lfloor \pi i_1 \rfloor + 1)^2 
    \le \frac{a^3}{3\pi}+\mleft(\frac{1}{2}+\frac{1}{\pi}\mright) a^2 + \mleft(1+\frac{\pi}{6}+\frac{1}{\pi} \mright) a 
    \le \mleft ( \frac{c}{\sqrt[3]{3\pi}} k_1 + 6.4 \mright)^3. 
\end{align}
Every tuple in $\mathcal{I}$ corresponds to at most one eigenvalue $\lambda_{i_1,i_2,i_3,\zeta}$ satisfying $\lambda_{i_1,i_2,i_3,\zeta} \ge -(c (k_1+1)-1.46)^2$, so there are at most $\mleft( \frac{c}{\sqrt[3]{3\pi}} k_1 + 6.4 \mright)^3$ such eigenvalues. Hence $k \le \mleft( \frac{c}{\sqrt[3]{3\pi}} k_1 + 6.4 \mright)^3$, so 
\begin{align}
    k_1 \ge \max \mleft\{ \frac{\sqrt[3]{3\pi}}{c} (\sqrt[3]{k}-6.4),1 \mright\}. \label{eq:3ball-k1bound}
\end{align}
Now, equations \eqref{eq:3ball-lambda} and \eqref{eq:3ball-k1def} imply $\lambda_{k,D} \le \lambda_{k_1,4k_1-1,2k_1,\zeta}$, for $\zeta \in \{0,\frac{\pi}{2}\}$. Thus, since $u_{k_1,4k_1-1,2k_1}$ has $8k_1^3$ nodal domains, \eqref{eq:3ball-k1bound} and the fact $k_1 \ge 1$ imply $r_k \ge 8k_1^3 \ge \max\Bigl\{\frac{24\pi}{c^3} (\sqrt[3]{k} - 6.4) )^3,8\Bigr\}\ge \max\{0.119 (\sqrt[3]{k}-6.4)^3,8\}$. {In this case, Pleijel's nodal domain theorem  (e.g.\ \cite[Theorem 1.2]{lena}) yields $\limsup_{k\to \infty} \frac{r_k}{k} \le \frac{9\pi}{2j_{0.5,1}^3} \approx 0.46$, i.e.\ these eigenfunctions with $\frac{r_k}{k}\ge 0.119+O(k^{-\frac{1}{3}})$ have within a factor of 4 of the maximum number of nodal domains possible for their position in the eigenfunction order.} 
Hence Theorem \ref{thm:cheeger} implies 
\begin{align}
    \lambda_{k,D} \ge \frac{1}{4} h_{r_k,D}^2 \ge \frac{1}{4} h_{\max\mleft\{\mleft\lceil 0.119 (\sqrt[3]{k}-6.4)^3 \mright\rceil,8\mright\}, D}^2. \label{eq:3ball-ndcheeger}
\end{align}
As in the previous example, one cannot apply Miclo's Theorem \ref{thm:miclocheeger} or \ref{thm:miclorcheeger} in this case, because $\mathbb{B}$ has non-empty boundary.

\subsection{Dynamic Cheeger constant on a cylinder with linear shear} \label{sec:shear}

Finally, we consider a linear shear on the cylinder $\mathcal{C}:=2\pi \mathbb{S}^1 \times [0,\pi]$, similarly to \cite[example 6.1]{F15}. We consider a dynamical system $\mathcal{T}$ as in definition \ref{def:dynsys}. We let $\mathrm{T}:=\{0,1,\ldots,t_{\max}\}$ for some even $t_{\max} \ge 2$, and for each $t$, we let $M_t:=\mathcal{C}$, and we define $g_t$ as the Euclidean metric and $V_t$ as two-dimensional Lebesgue measure. For some $b>0$, we define each $\Phi^{(t)}:\mathcal{C} \to \mathcal{C}$ by
\begin{align}
    \Phi^{(t)}(x, y) := \mleft(x + b \frac{t}{t_{\max}} y \imod{2\pi}, y\mright). \label{eq:shear-phi}
\end{align}
The dynamics $\Phi^{(t)}$ represents linear shear in the $x$-coordinate on the cylinder. The functions 
\begin{align}
    u_{k_1,k_2,\zeta}^d(x,y)&:=\cos\mleft(k_1 \mleft(x + \zeta - \frac{b}{2} y \mright)\mright) \cos(k_2 y), \label{eq:shear-efunc}
\end{align}
for $k_1,k_2=0,1,2,\ldots$, and $\zeta \in \{0,\frac{\pi}{2}\}$, taking $\zeta=0$ whenever $k_1=0$, are a complete basis of eigenfunctions for $\Delta^d$ under dynamic Neumann boundary conditions. 
This follows since for each $t \in \mathrm{T}$, writing $\tilde{x}_t:=x+\zeta+b\mleft(\frac{t}{t_{\max}}-\frac{1}{2}\mright)y$ for brevity, we have $\Phi^{(t)}_* u_{k_1,k_2,\zeta}^d(x,y)=\cos\mleft(k_1 \tilde{x}_t \mright)\cos(k_2y)$, so
\begin{align*}
    &\Delta\Phi^{(t)}_* u_{k_1,k_2,\zeta}^d(x,y) \\
    &{= -\frac{\partial}{\partial x}} \mleft[k_1 \sin\mleft(k_1 \tilde{x}_t \mright)\cos(k_2y) \mright] 
    - \frac{\partial}{\partial y} \Bigl[ k_1 b \mleft(\tfrac{t}{t_{\max}}-\tfrac{1}{2}\mright) \sin\mleft( k_1 \tilde{x}_t \mright) \cos(k_2 y) 
+ k_2 \cos \mleft(k_1 \tilde{x}_t\mright)\sin(k_2y)\Bigr] \\
    &{= -\mleft(k_1^2 \mleft(1+b^2\mleft(\tfrac{t}{t_{\max}}-\tfrac{1}{2}\mright)^2\mright)+k_2^2\mright) \Phi^{(t)}_* u_{k_1,k_2,\zeta}^d(x,y)} 
    + 2k_1k_2b\mleft(\tfrac{t}{t_{\max}}-\tfrac{1}{2}\mright) \sin\mleft(k_1 \tilde{x}_t \mright) \sin(k_2 y).
\end{align*}
Then, since $\sum_{t=0}^{t_{\max}} \mleft(\frac{t}{t_{\max}}-\frac{1}{2}\mright)=0$ and 
$\sum_{t=0}^{t_{\max}} \mleft(\frac{t}{t_{\max}}-\frac{1}{2}\mright)^2
=\frac{(t_{\max}+1)(t_{\max}+2)}{12 t_{\max}}
=\frac{|\mathrm{T}|(|\mathrm{T}|+1)}{12(|\mathrm{T}|-1)}$, 
we have 
\begin{align}
    &\Delta^d u_{k_1,k_2,\zeta}(x,y) \nonumber \\
    &{=-\tfrac{1}{|\mathrm{T}|}} \sum_{t=0}^{t_{\max}} 
    \biggl[\mleft(k_1^2 \mleft(1 + b^2 \mleft(\tfrac{t}{t_{\max}}-\tfrac{1}{2}\mright)^2 \mright) + k_2^2\mright) u_{k_1,k_2,\zeta}^d 
    + 2k_1k_2b\mleft(\tfrac{t}{t_{\max}}-\tfrac{1}{2}\mright) \sin\mleft(k_1 \tilde{x}_0 \mright) \sin(k_2 y) \biggr] \nonumber \\
    &{=-\mleft(k_1^2 \mleft(1 + \tfrac{b^2 (|\mathrm{T}|+1)}{12 (|\mathrm{T}|-1)}\mright) + k_2^2 \mright) u_{k_1,k_2,\zeta}^d}, \nonumber
\end{align} 
i.e.\ each $u_{k_1,k_2,\zeta}^d$ is an eigenfunction with eigenvalue $\lambda_{k_1,k_2,\zeta}^d := -k_1^2 \bigl(1+\frac{b^2 (|\mathrm{T}|+1)}{12 (|\mathrm{T}|-1)}\bigr)- k_2^2$. 
These eigenfunctions form a complete orthogonal Hilbert basis for $L^2(\mathcal{C},\mathrm{Leb}^{{2}})$, since for $t^*=\frac{t_{\max}}{2}$, the $L^2$-isometry $\Phi^{(t^*)}_*:L^2(\mathcal{C}) \to L^2(\mathcal{C})$ sends the functions \eqref{eq:shear-efunc} to the complete orthogonal Hilbert basis \eqref{eq:cylinder-efunc}. 

To apply Theorem \ref{thm:dcheeger}, we need a lower bound for $r_k$ for each sufficiently large $k$. We consider $k \ge \pi pq + \sqrt{2}(p+2q)+1$, where $p \ge q \ge 1$ are integers for which $1+\frac{b^2(|\mathrm{T}|+1)}{12(|\mathcal{T}-1)}=\frac{p^2}{q^2}$. Then each eigenvalue $\lambda_{k_1,k_2,\zeta}^d$ can be written 
\begin{align}
    \lambda_{k_1,k_2,\zeta}^d = -\frac{p^2}{q^2}k_1^2 - k_2^2. \label{eq:shear-evalue}
\end{align}
We obtain our bound on $r_k$ in the following steps. 
First, we show that for $k_1\in \{q,2q,\ldots\}$, the eigenfunctions $u_{k_1,\frac{p}{q}k_1,0}^d$ and $u_{k_1,\frac{p}{q}k_1,\frac{\pi}{2}}^d$ have the maximum number of nodal domains, among eigenfunctions of the form \eqref{eq:shear-efunc} with eigenvalue $\ge -2\frac{p^2}{q^2}k_1^2$. 
Second, for each $k_1\in \{q,2q,\ldots\},$ we obtain an upper bound for 
\begin{align}
    \mathcal{E}(k_1):=\#\left\{\lambda_{i_1,i_2,\zeta}^d: \lambda_{i_1,i_2,\zeta}^d \ge -2\frac{p^2}{q^2}k_1^2\right\}, \label{eq:shear-edef}
\end{align}
the number of eigenvalues $\lambda_{i_1,i_2,\zeta}^d$ (with multiplicity) satisfying $\lambda_{i_1,i_2,\zeta}^d \ge -2\frac{p^2}{q^2}k_1^2$, and hence put an upper bound on the position of $\lambda_{k_1,\frac{p}{q}k_1,0}^d$ in the eigenvalue ordering. 
Third, we use this bound to show that for each $k \ge \pi pq + \sqrt{2}(p+2q)+1=(\pi q^2 +\sqrt{2}q)\sqrt{1+\frac{b^2(|\mathrm{T}|+1)}{12(|\mathrm{T}|-1)}}+2\sqrt{2}q+1$, there is some $k_1\in \{q,2q,\ldots\}$ such that $\lambda_{k_1,\frac{p}{q} k_1, 0}^d \ge \lambda_{k,N}^d$, and also to bound the largest such $k_1$ from below. 
Finally, for this $k$ and $k_1$, we use the number of nodal domains in $u_{k_1,\frac{p}{q}k_1,0}^d$ to give a lower bound on $r_k$, and hence we use Theorem \ref{thm:dcheeger} to bound $\lambda_{k,N}^d$ in terms of $h_{r_k,N}^d$. 

\emph{Step 1:} We begin by proving that $u_{k_1,\frac{p}{q}k_1,0}^d$ and $u_{k_1,\frac{p}{q}k_1,\frac{\pi}{2}}^d$ have the maximal number of nodal domains among eigenfunctions $u_{i_1,i_2,\zeta}^d$ of the form \eqref{eq:shear-efunc} for which $\lambda_{i_1,i_2,\zeta}^d \ge -2\frac{p^2}{q^2}k_1^2$. Each eigenfunction $u_{i_1,i_2,\zeta}^d$ has $\max\{2i_1,1\}(i_2+1)$ nodal domains by \eqref{eq:shear-efunc} (since $\cos\mleft(i_1\mleft(x+\zeta -\frac{b}{2}y\mright)\mright)$ has $\max\{2i_1,1\}$ nodal domains and $\cos(i_2y)$ has $i_2+1$ nodal domains). Thus, by \eqref{eq:shear-evalue}, maximising the number of nodal domains in $u_{i_1,i_2,\zeta}^d$ subject to $\lambda_{i_1,i_2,\zeta}^d \ge -2\frac{p^2}{q^2}k_1^2$ is equivalent to solving $\max\{2i_1(i_2+1):(i_1,i_2) \in \mathbb{Z}_{>0}, -\frac{p^2}{q^2}i_1^2 - i_2^2 \ge -2\frac{p^2}{q^2}k_1^2\}$. By a similar relaxation argument to section \ref{sec:cylinder}, this is uniquely maximised by $(i_1,i_2)=(k_1,\frac{p}{q}k_1)$. Hence eigenfunctions $u_{k_1,\frac{p}{q}k_1,\zeta}^d$ for $\zeta \in \{0,\frac{\pi}{2}\}$ maximise the number of nodal domains, among eigenfunctions $u_{i_1,i_2,\zeta}^d$ of the form \eqref{eq:shear-efunc} satisfying $\lambda_{i_1,i_2,\zeta}^d \ge -2\frac{p^2}{q^2}k_1^2$. 

\emph{Step 2:} Choose any $k_1=q,2q,\ldots$. We can bound $\mathcal{E}(k_1)$ (defined in \eqref{eq:shear-edef}) by considering three cases: eigenvalues $\lambda_{i_1,i_2,\zeta}^d$ with $i_1,i_2\ge 1$, eigenvalues $\lambda_{i_1,0,\zeta}^d$ with $i_1 \ge 1$, and eigenvalues $\lambda_{0,i_2,0}^d$ for $i_2 \ge 0$. 

The set $\{\lambda_{i_1,i_2,\zeta}:\lambda_{i_1,i_2,\zeta}\ge -2\frac{p^2}{q^2}k_1^2,i_1,i_2 \ge 1\}$ is in bijection with the set $\{(i_1,i_2,\zeta):\zeta \in \{0,\frac{\pi}{2}\},(i_1,i_2) \in \mathbb{Z}_{>0}, -\frac{p^2}{q^2}i_1^2 - i_2^2 \ge -2\frac{p^2}{q^2}k_1^2\}$, by \eqref{eq:shear-evalue}. These tuples $(i_1,i_2)$ are in bijection with the grid points $(i_1,i_2)$ in the positive quadrant $Q_{pq}$ of the ellipse $\frac{x^2}{2k_1^2}+\frac{q^2 y^2}{2p^2 k_1^2} \le 1$. The quadrant $Q_{pq}$ has area $\frac{\pi p}{2q}k_1^2$, and each grid point $(i_1,i_2) \in Q_{pq}$ with $i_1,i_2 \ge 1$ is associated with a unit area in $Q_{pq}$. Therefore, there are at most $\frac{\pi p}{2q}k_1^2$ grid points $(i_1,i_2)$, so there are at most $\frac{\pi p}{q}k_1^2$ tuples $(i_1,i_2,\zeta)$, and hence at most $\frac{\pi p}{q} k_1^2$ eigenvalues $\lambda_{i_1,i_2,\zeta}^d \ge -2\frac{p^2}{q^2}k_1^2$ with $i_1,i_2 \ge 1$. 

By \eqref{eq:shear-evalue}, the eigenvalues $\lambda_{i_1,0,\zeta}^d \ge -2\frac{p^2}{q^2}k_1^2$ with $i_1 \ge 1$ are in bijection with the tuples $(i_1,\zeta)$ with $i_1 \in \mathbb{Z} \cap [1,\sqrt{2}k_1]$ and $\zeta \in \{0,\frac{\pi}{2}\}$, so there are $2\lfloor \sqrt{2}k_1 \rfloor$ such eigenvalues. Similarly, the eigenvalues $\lambda_{0,i_2,0}^d \ge -2\frac{p^2}{q^2} k_1^2$ are in bijection with the integers $i_2 \in \mathbb{Z} \cap [0,\sqrt{2}\frac{p}{q}k_1]$, so there are $\lfloor \sqrt{2}\frac{p}{q}k_1 \rfloor+1$ such eigenvalues. Combining these three cases, the number $\mathcal{E}(k_1)$ of eigenvalues $\lambda_{i_1,i_2,\zeta}^d \ge -2\frac{p^2}{q^2}k_1^2$, counted with multiplicity, is bounded above by 
\begin{align}
    \mathcal{E}(k_1) \le \frac{\pi p}{q}k_1^2 + 2\lfloor \sqrt{2}k_1 \rfloor + \mleft\lfloor \frac{\sqrt{2}p}{q}k_1 \mright\rfloor + 1. \label{eq:shear-kbound}
\end{align}

\emph{Step 3:} 
Equations \eqref{eq:shear-evalue}-\eqref{eq:shear-edef} imply there are no more than $\mathcal{E}(1)$ eigenvalues $\ge \lambda_{q,p,0}^d$, so \eqref{eq:shear-kbound} implies there are no more than $\pi pq + \sqrt{2}(p+2q)+1$ such eigenvalues. 
Hence for each $k \ge \pi pq + \sqrt{2}(p+2q) + 1$, we have $\lambda_{q,p,0}^d \ge \lambda_{k,N}^d$, i.e.\ for $\tilde{k}_1=q$ we have $\lambda_{\tilde{k}_1,\frac{p}{q}\tilde{k}_1,0}^d \ge \lambda_{k,N}^d$. 
Define 
\begin{align}
    k_1 := \max \mleft\{\tilde{k}_1 \in \{q,2q,\ldots\} : \lambda_{\tilde{k}_1,\frac{p}{q}\tilde{k}_1,0}^d \ge \lambda_{k,N}^d \mright\}, \label{eq:shear-kdef}
\end{align}
then each multiple $\tilde{k}_1$ of $q$ greater than $k_1$ satisfies $\lambda_{k,N}^d \ge \lambda_{\tilde{k}_1,\frac{p}{q}\tilde{k}_1,0}^d$. 
In particular, by \eqref{eq:shear-evalue}, we have $\lambda_{k,N}^d \ge \lambda_{k_1+q,\frac{p}{q}(k_1+q),0}^d=-2\frac{p^2}{q^2}(k_1+q)^2$. 
Therefore, since $\lambda_{k,N}^d$ is the $k$th-smallest eigenvalue in absolute value, \eqref{eq:shear-edef} implies $k \le \mathcal{E}(k_1+q)$. 
Then \eqref{eq:shear-kbound} yields $k \le \frac{\pi p}{q}(k_1+q)^2 + 2 \sqrt{2} (k_1+q) + \frac{\sqrt{2}p}{q} (k_1+q)+1$. Applying the quadratic formula yields $k_1 \ge \sqrt{\frac{qk}{\pi p} - \frac{q}{\pi p} + \frac{1}{2\pi^2}(\frac{2q}{p}+1)^2} - (\frac{1}{\sqrt{2}\pi} + \frac{\sqrt{2}q}{\pi p} + q)$. Noting that $\frac{1}{2\pi^2} (\frac{2q}{p}+1)^2 - \frac{q}{\pi p} > \frac{1}{2\pi^2} (\frac{2q}{p}-1)^2>0+
$, we obtain 
\begin{align}
    k_1 \ge \sqrt{\frac{q k}{\pi p}} - \mleft(\frac{1}{\sqrt{2}\pi} + \frac{\sqrt{2}q}{\pi p} + q \mright). \label{eq:shear-k1bound}
\end{align}

\emph{Step 4:} Choose $k$ and $k_1$ as in step 3, so that $\lambda_{k_1,\frac{p}{q}k_1,0}^d \ge \lambda_{k,N}^d$ by \eqref{eq:shear-kdef}. Then the number of nodal domains in $u_{k_1,\frac{p}{q}k_1,0}^d$ gives a lower bound on $r_k$. This eigenfunction has $2k_1(\frac{p}{q}k_1+1)$ nodal domains by the reasoning in step 1, so $r_k \ge 2k_1(\frac{p}{q}k_1+1)$. Substituting \eqref{eq:shear-k1bound} into this expression gives $r_k \ge 2\mleft(\sqrt{\frac{q k}{\pi p}} - \mleft(\frac{1}{\sqrt{2}\pi} + \frac{\sqrt{2}q}{\pi p} + q \mright)\mright)\mleft(\sqrt{\frac{p k}{\pi q}} - \frac{p}{q}\mleft(\frac{1}{\sqrt{2}\pi} + \frac{\sqrt{2}q}{\pi p} + q \mright)+1\mright)$. Expanding and noting that $p \ge q \ge 1$ 
so $2\sqrt{\frac{qk}{\pi p}}\mleft(1-\frac{2\sqrt{2}}{\pi}\mright) > 0$, $\frac{2\sqrt{2}p}{\pi}+\frac{4\sqrt{2}q}{\pi}>2q+\frac{\sqrt{2}}{\pi}$ and $\frac{p}{\pi^2 q}+\frac{4q}{\pi^2 p} + \frac{4}{\pi^2} > \frac{2\sqrt{2}q}{\pi p}$, we obtain $r_k \ge \frac{2k}{\pi}-2.8\sqrt{pqk}+2pq$. Then the definition of $p$ and $q$ before \eqref{eq:shear-evalue} implies $r_k \ge \frac{2k}{\pi}-2.8q\sqrt[4]{1+\frac{b^2(|\mathrm{T}|+1)}{12(|\mathrm{T}|-1)}}\sqrt{k}+2q^2\sqrt{1+\frac{b^2(|\mathrm{T}|+1)}{12(|\mathrm{T}|-1)}}$. Substituting $k_1 \ge q$ into $r_k\ge 2k_1(\frac{p}{q}k_1+1)$ instead, we additionally obtain $r_k \ge 2q(p+1)=2q^2\sqrt{1+\frac{b^2(|\mathrm{T}|+1)}{12(|\mathrm{T}|-1)}}+2q$. Hence, rewriting the definition of $k$ from step 3 using the definition of $p$ and $q$ before \eqref{eq:shear-evalue}, Theorem \ref{thm:dcheeger} implies that for each $k \ge (\pi q^2 +\sqrt{2}q)\sqrt{1+\frac{b^2(|\mathrm{T}|+1)}{12(|\mathrm{T}|-1)}}+2\sqrt{2}q+1$, we have 
\begin{align}
    \lambda_{k,N}^d \le -\frac{1}{4} (h_{r_k}^d)^2 
    \le -\frac{1}{4} \mleft(h_{\max\mleft\{\left\lceil \frac{2k}{\pi}
      - 2.8q\sqrt[4]{1+\frac{b^2(|\mathrm{T}|+1)}{12(|\mathrm{T}|-1)}}\sqrt{k}
     + 2q^2\sqrt{1+\frac{b^2(|\mathrm{T}|+1)}{12(|\mathrm{T}|-1)}} \right\rceil, 
    2q^2\sqrt{1+\frac{b^2(|\mathrm{T}|+1)}{12(|\mathrm{T}|-1)}}+2q \mright\},N}^d\mright)^2.
\end{align}
Asymptotically for large $k$, this bound becomes $\lambda_{k,N}^d \le -\frac{1}{4} (h_{\frac{2k}{\pi}-O(\sqrt{k}),N}^d)^2$, irrespective of the shear strength $b$ and number of time steps $|\mathrm{T}|$. 
Pre-asymptotically for intermediate-sized $k$, this bound links $\lambda_{k,N}^d$ to $h_{j,N}^d$ for progressively smaller $j$ as the shear strength increases. 
This is because the domain behaves like a cylindrical domain with progressively more mismatched sides, so that gridlike packings of $\mathcal{C}$ with the optimal aspect ratio for each packing element are rarer. 
We cannot apply Theorems \ref{thm:dmiclo} or \ref{thm:dmiclor}, our dynamic versions of Theorems \ref{thm:miclocheeger} and \ref{thm:miclorcheeger}, because $\mathcal{C}$ has non-empty boundary.

%% file: 06_summary.tex
\section{Summary}

The sequence of the $k$th (Neumann or Dirichlet) Cheeger constants for a weighted Riemannian manifold (Definition \ref{def:whcheeg}) and the corresponding $k$-packings with small Cheeger ratio (Definition \ref{def:packing}) together give a global geometric description of weighted Riemannian manifolds. 
There are no existing algorithms for computing $k$-packings for $k \ge 2$ with small Cheeger ratio on arbitrary Riemannian manifolds. 
We proposed some methods for obtaining upper bounds on the Cheeger constants, and for finding packings with quality guarantees, i.e.\ upper bounds on their Cheeger ratios (Theorem \ref{thm:cheeger} and Proposition \ref{thm:seba}). 
We showed that for any Neumann or Dirichlet eigenfunction, its eigenvalue gives an upper bound on the Cheeger constant corresponding to the number of nodal domains in the eigenfunction (Theorem \ref{thm:cheeger}). 
Moreover, we showed that positive-measure collections of the superlevel sets within each nodal domain give rise to packings whose Cheeger ratios are bounded above in terms of the eigenvalue{.} 
This bound is straightforward to compute, but it only produces $k$-packings from eigenfunctions with $k$ nodal domains. 
Sometimes, it is possible to combine geometric information from several eigenfunctions to obtain more features than the number of nodal domains in any single eigenfunction. One obtains disjointly supported functions, each supported on a single feature, by taking linear combinations of eigenfunctions and applying soft thresholding. 
The sparse eigenbasis approximation (SEBA) algorithm \cite{FRS19} can be used to find suitable linear combinations. 
We showed that if the separation into disjointly supported sparse functions is successful, then positive-measure collections of the resulting superlevel sets yield packings with an upper bound on their Cheeger ratios (Proposition \ref{thm:seba}). 
This bound depends only on the largest eigenvalue (in absolute value) and the effectiveness of the separation (i.e.\ the fraction of the $L^2$ mass of the linear combinations that is preserved by the thresholding operation). 

Coherent sets in nonautonomous dynamical systems are sets with small dynamic Cheeger ratio (Definition \ref{def:dcheeg}). 
We showed that positive-measure collections of the superlevel sets within each nodal domain of a dynamic Laplacian eigenfunction yield packings consisting of coherent sets, i.e.\ packings whose dynamic Cheeger ratios are bounded above (Theorem \ref{thm:dcheeger}).
Also, as in the static case, it is sometimes possible to obtain more coherent sets than the number of nodal domains in any single eigenfunction, by taking linear combinations of the first $k$ eigenfunctions and applying soft thresholding. We showed (Proposition \ref{thm:dseba}) that positive-measure collections of the resulting superlevel sets have their dynamic Cheeger ratios bounded above in terms of the largest eigenvalue (in absolute value), and the effectiveness of the separation (fraction of $L^2$ mass preserved by soft thresholding).